\documentclass[twoside,12pt]{article}

\usepackage{amsmath}
\usepackage{amssymb}
\usepackage{amscd}
\usepackage{amsthm}

\numberwithin{equation}{section}

\theoremstyle{plain}
\newtheorem{theorem}{Theorem}[section]

\newtheorem{Prop}[theorem]{Proposition}
\newtheorem{Le}[theorem]{Lemma}
\newtheorem{Cor}[theorem]{Corollary}

\theoremstyle{remark}

\theoremstyle{definition}

\newcommand{\R}{\mathbb R}

\def\ga{\gamma}

\def\ra{\rightarrow}

\def\e{\emph}
\def\i{\infty}
\def\p{\partial}
\def\b{\begin}

\begin{document}

\title{
{A Rigidity  Property of Some Negatively Curved Solvable Lie
Groups}}
\author{Nageswari Shanmugalingam, Xiangdong Xie}
\date{  }

\maketitle

\begin{abstract}
  We show that for some negatively curved solvable Lie groups, all
self quasiisometries are almost isometries. We prove this by showing
that all self quasisymmetric maps of the ideal boundary (of the
solvable Lie groups)  are bilipschitz with respect to the visual
metric. We also define parabolic visual metrics on the ideal
boundary of Gromov hyperbolic spaces and relate them  to visual
metrics.
\end{abstract}




{\bf{Keywords.}} quasiisometry, quasisymmetric map, negatively
curved solvable Lie groups.



 {\small {\bf{Mathematics Subject
Classification (2000).}} 20F65,  30C65, 53C20.






\setcounter{section}{0} \setcounter{subsection}{0}

\section{Introduction}\label{s0}

In recent years, there have been a lot of interest in the large
scale geometry of solvable Lie groups and finitely generated
solvable groups (\cite{D}, \cite{EFW1}, \cite{EFW2}, \cite{FM1},
\cite{FM2}, \cite{FM3}, \cite{Pe}). In particular,
Eskin, Fisher and Whyte (\cite{EFW1}, \cite{EFW2}) proved the quasiisometric
rigidity of the 3-dimensional solvable Lie group Sol.
In this paper, we use quasiconformal analysis to prove a
rigidity property of some negatively curved solvable Lie groups.

Let $A$ be an $n\times n$ diagonal matrix with  real  eigenvalues
${\alpha_i}$   with $\alpha_{i+1}>\alpha_i>0$:
\[A=\left(\begin{array}{cccc}
{\alpha_1}I_{n_1}  &  {\bf{0}} & \cdots & {\bf{0}}\\
{\bf{0}}  &  {\alpha_2}I_{n_2}  & \cdots & {\bf{0}}\\
\cdots & \cdots &  \cdots  &  \cdots  \\
{\bf{0}} & {\bf{0}} & \cdots  &
{\alpha_r}I_{n_r}\end{array}\right),\]
where $I_{n_i}$ is the $n_i\times n_i$ identity matrix and the ${\bf{0}}$'s
are zero matrices (of various sizes).
Let $\R$  act on $\R^n$ by the linear transformations
$e^{tA}$ ($t\in \R$)    and we can form the semidirect product
$G_A=\R^n\rtimes_A \R$. That is, $G_A=\R^n\times\R$
as a smooth manifold, and the group operation is given
for all $(x,t), (y, s)\in\R^n\times\R$ by:
\[
    (x,t)\cdot (y, s)=(x+e^{tA}y, t+s).
\]
The group $G_A$  is a
simply connected  solvable Lie group and is the
subject of study in this paper.

We endow $G_A$ with the left invariant metric determined by taking
the standard Euclidean metric at the identity of
$G_A\approx\R^n\times\R=\R^{n+1}$.
With this  metric  $G_A$ has  sectional curvature $-\alpha_r^2\le
K\le - \alpha_1^2$ (and so is Gromov hyperbolic). Hence $G_A$ has a
well defined ideal boundary $\p G_A$. There is a so-called cone
topology on $\overline{G_A}=G_A\cup \p G_A$, in which $\p G_A$ is
homeomorphic to the $n$-dimensional sphere and $\overline{G_A}$ is
homeomorphic to the closed $(n+1)$-ball in the Euclidean space. For
each $x\in \R^n$,  the map $\gamma_x: \R\ra G_A$,
$\gamma_x(t)=(x,t)$  is a geodesic. We call such a geodesic a
vertical geodesic.  It can be checked  that all vertical geodesics
are asymptotic as $t\ra +\infty$. Hence they define a point $\xi_0$
in the ideal boundary $\p G_A$.

Since  $G_A$  is Gromov hyperbolic,  there is a family of visual
metrics on $\p G_A$. For each $\xi\in\p G_A$, there is also the
so-called parabolic visual metric on $\p G_A\backslash\{\xi\}$. The
relation between visual metrics and parabolic visual metrics is
analogous to the relation between spherical metric (on the sphere)
and the Euclidean metric (on the one point complement of the
sphere). See Section~\ref{parabolic} for a discussion   of all these
in the setting of Gromov hyperbolic spaces. We next recall
the parabolic  visual  metric $D$ on $\partial G_A$ viewed from $\xi_0$.

The set  $\p G_A\backslash\{\xi_0\}$ can be naturally identified
with $\R^n$ (see Section \ref{s1}).
Write $\R^n=\R^{n_1}\times \cdots \times \R^{n_r}$, where $\R^{n_i}$
is the eigenspace associated to the eigenvalue $\alpha_i$ of $A$.
Each point $x\in \R^n$ can be written as $x=(x_1, \cdots, x_r)$ with
$x_i\in \R^{n_i}$.     
 The parabolic visual metric  $D$ on   $\p
 G_A\backslash\{\xi_0\}\approx \R^n$ is defined by:
\[
  D(x,y)=
    \max\{|x_1-y_1|,|x_2-y_2|^{\alpha_1/\alpha_2},\cdots,|x_r-y_r|^{\alpha_1/\alpha_r}\},
\]
for  all $x=(x_1,\cdots,x_r), y=(y_1, \cdots, y_r)\in\R^n$.

Let $\eta: [0,\i)\ra [0,\i)$ be a homeomorphism.
An embedding  of metric spaces
$f:X\to Y$ is an
\e{$\eta$-quasisymmetric embedding} if for all distinct triples
$x,y,z\in X$, we have
\[
   \frac{d(f(x), f(y))}{d(f(x), f(z))}\le \eta\left(\frac{d(x,y)}{d(x,z)}\right).
\]
If $f$ is further assumed to be a  homeomorphism, we say it is
\e{$\eta$-quasisymmetric}.
A map $f:X\to Y$ is quasisymmetric if it is $\eta$-quasisymmetric
for some $\eta$.

 When $r\ge 2$,  Bruce Kleiner  has proved that (\cite{K})
   every self quasisymmetry of $\p G_A$ (equipped with a
   visual   metric)  preserves the horizontal foliation (see Section 3)
       and
     fixes the point $\xi_0$.
  This is one of the main ingredients in the proof of our main
  result.  Since Kleiner's   proof  is unpublished, we include a
  proof here for completeness.
 Notice that Kleiner's result implies that
  a self quasisymmetry of  $\p G_A$
  induces a self map of $(\R^n, D)$.

The following is the main result of this paper.

\begin{theorem}\label{intromain}
Let $G_A$  and $\xi_0\in \p G_A$ be as above. If $r\ge 2$, then
   every self quasisymmetry of $\p G_A$ (equipped with a
   visual   metric)
        is bilipschitz  on
   $\partial G_A\setminus\{\xi_0\}$
        with respect to the parabolic visual metric
$D$.  
\end{theorem}

 One should compare this with quasiconformal maps on
  Euclidean spaces (\cite{GV})  and Heisenberg groups  (\cite{B}), where there are
       non-bilipschitz
quasiconformal  maps. On the other hand, the conclusion of
Theorem~\ref{intromain} is not as strong as in the cases of
quaternionic
  hyperbolic spaces,
Cayley plane (\cite{P2}) and Fuchsian buildings (\cite{BP},  \cite{X}),
where every quasisymmetric map of the ideal boundary is actually a
conformal map. In our case, there are many  non-conformal
quasisymmetric maps of the ideal boundary of $G_A$.
We also remark that in~\cite[Section~15]{T2} Tyson has previously classified
(quasi)metric spaces of the form $(\R^n, D)$ up to quasisymmetry.

We list three consequences of Theorem       \ref{intromain}.

Let $L\ge 1$ and $C\ge 0$. A (not necessarily continuous ) map
$f:X\ra Y$ between two metric spaces is an
$(L,A)$-\e{quasiisometry} if:
\newline (1) $d(x_1,x_2)/L-C\le d(f(x_1), f(x_2))\le L\, d(x_1, x_2)+C$
for all $x_1,x_2\in X$;
\newline (2)  for any $y\in Y$, there is some $x\in X$ with
$d(f(x), y)\le C$.
\newline
In the case   $L=1$, we call  $f$  an \e{almost isometry}.

\begin{Cor}\label{c0}
Assume that $r\ge 2$. Then every self quasiisometry of $G_A$ is an
almost isometry.
\end{Cor}

Notice that an almost isometry is not necessarily a finite distance
away from an isometry.

The following result was previously obtained by B.   Kleiner
\cite{K}.

\begin{Cor}\label{finitege}
If $r\ge 2$, then $G_A$ is not
quasiisometric to any finitely generated group.
\end{Cor}

In the identification of $G_A$ with $\R^n\times \R$,  we view the
map $h:\R^n\times \R$, $h(x,t)=t$ as the height function.
A quasiisometry $\varphi$ of $G_A$ is height-respecting if 
$|h(\varphi(x,t))-t|$ is bounded independent of $x,t$.

\begin{Cor}\label{c2}
Assume that $r\ge 2$.  Then all self
quasiisometries of $G_A$ are height-respecting.
\end{Cor}

The question of whether a quasiisometry of $G_A$ is
height-respecting is important for the following three reasons.
First, Mosher and Farb (\cite{FM1}) have classified a large class of
solvable Lie groups (including groups of type $G_A$) up to
height-respecting quasiisometries. Second, there is no known
examples of non-height-respecting quasiisometries except for rank
one symmetric spaces of noncompact type. Finally, showing a
quasiisometry is height-respecting is a  key
     step in the proof of
the quasiisometric rigidity of Sol ({\cite{EFW1}, \cite{EFW2}).

When $r=1$, the group $G_A$ is isometric to a rescaling of the real hyperbolic
space. In this case, all the above results fail.

This paper is structured as follows.
   In Section \ref{s1} we review some basics about the group $G_A$.
In Section~3 we prove that quasisymmetric self-maps of $\p
G_A\setminus\{\xi_0\}$ equipped with the parabolic visual metric
preserve horizontal foliations, and in Section~4 we will prove that
such maps are bilipschitz with respect to this metric. The main
result of this paper, Theorem~\ref{intromain}, is proven in
Section~5, where a discussion of  parabolic visual metrics on the
ideal
  boundary
 and their connection to the
visual metrics 
   can
also be found. In Section~6 we provide the proofs of the Corollaries
stated in Section~1.

\noindent {\bf{Acknowledgment}}. {We would like to thank Bruce
Kleiner for helpful discussions.  The second author would also like
to thank the Department of Mathematical Sciences at Georgia Southern
University    for generous travel support.}

\section{The Solvable Lie Groups  $G_A$}\label{s1}

In this section we review some basic facts about the group $G_A$
 and define several parabolic visual (quasi)metrics on the
 ideal boundary.

Let $A$ and $G_A$
  be as in the Introduction.
We endow $G_A$ with the left invariant metric determined by taking
the standard Euclidean metric at the identity of
$G_A\approx\R^n\times\R=\R^{n+1}$.
At a point $(x,t)\in\R^n\times\R\approx G_A$, the tangent space is
identified with $\R^n\times\R$, and the Riemannian metric is given
by the symmetric matrix
\[
  \left(\begin{array}{cc}e^{-2tA} & 0\\ 0 & 1\end{array}\right).
\]
With this    metric  $G_A$ has  sectional curvature $-\alpha_r^2\le
K\le - \alpha_1^2$. Hence $G_A$ has a well defined ideal boundary
$\p G_A$.   All vertical geodesics  $\gamma_x$ ($x\in \R^n$)  are
asymptotic as $t\ra +\infty$. Hence they define a point $\xi_0$ in
the ideal boundary $\p G_A$.

    The sets $\R^n\times\{t\}$  ($t\in \R$)
are horospheres centered at $\xi_0$.
      For each  $t\in \R$, the
induced metric on  the horosphere $ \R^n\times\{t\}\subset G_A$   is
determined by the quadratic form $e^{-2tA}$. This metric has
distance formula $d_{\R^n\times \{t\}}((x,t), (y,t))=|
e^{-tA}(x-y)|$.  Here $|\cdot |$ denotes the Euclidean norm.
   The
distance between two horospheres, corresponding to $t=t_1$ and
$t=t_2$, is $|t_1-t_2|$. It follows that for $(x_1,t_1),(x_2,t_2)\in
G_A=\R^n\times\R$,
\begin{equation}\label{eq:lowerbound}
  d((x_1,t_1),(x_2,t_2))\ge |t_1-t_2|.
\end{equation}

Each geodesic ray in $G_A$ is  asymptotic to either  an upward
oriented vertical geodesic or a downward oriented vertical geodesic.
The upward oriented geodesics are asymptotic to $\xi_0$ and the
downward oriented vertical  geodesics are in 1-to-1 correspondence
with $\R^n$. Hence $\p G_A\backslash\{\xi_0\}$ can be naturally
identified with $\R^n$.

Given $x, y\in\R^n\approx \p G_A\backslash\{\xi_0\}$, the  parabolic
visual quasimetric $D_e(x,y)$ is defined as follows:
${D}_e(x,y)=e^t$, where $t$ is the unique real number such that at
height $t$ the two vertical geodesics $\ga_x$ and $\ga_y$ are at
distance one apart in the horosphere; that is, $d_{\R^n\times
\{t\}}((x,t), (y,t))=| e^{-tA}(x-y)|=1.$  Here the subscript
{\it{e}} in $D_e$  means it corresponds to the Euclidean norm. This
definition of parabolic visual  quasimetric is very natural, but
$D_e$ does not have a simple formula. Next we describe another
parabolic visual quasimetric   which  is bilipschitz equivalent with
$D_e$ and admits a simple formula.
 Recall that a quasimetric on a set $A$ is a function $\rho:
    A \times A\rightarrow [0, \i)$ satisfying: (1) $\rho(x,y)=\rho(y,x)$ for
    all $x, y\in A$; (2) $\rho(x,y)=0$ only when $x=y$; (3) there
    is a constant  $L\ge 1$ such that $\rho(x, z)\le
    L(\rho(x,y)+\rho(y,z))$ for all $x, y,z\in A$.

 In addition to the Euclidean norm,
   there is  another norm on $\R^n$ that is naturally associated to   $G_A$.
Write $\R^n=\R^{n_1}\times \cdots \times \R^{n_r}$, where $\R^{n_i}$
is the eigenspace associated to the eigenvalue $\alpha_i$ of $A$.
Each point $x\in \R^n$ can be written as $x=(x_1, \cdots, x_r)$ with
$x_i\in \R^{n_i}$.
    The block supernorm is  given by:
   $|x|_s=\max\{|x_1|, \cdots, |x_r|\}$  for  $x=(x_1, \cdots, x_r)$.   Using this norm one can
   define another parabolic visual quasimetric on
$\p G_A \backslash \{\xi_0\}$ as follows: ${D}_s(x,y)=e^t$, where
$t$ is the unique real number such that at height $t$ the two
vertical geodesics $\ga_x$ and $\ga_y$ are at distance one apart
with respect to the norm $|\cdot|_s$; that is, $|
e^{-tA}(x-y)|_s=1.$   Here the subscript {\it{s}} in  $D_s$  means
it corresponds to the block supernorm  $|\cdot|_s$. Then    $D_s$
is given by \cite[Lemma~7]{D}:
\[
   {D_s}(x,y)=\max\{|x_1-y_1|^{\frac{1}{\alpha_1}},
  \cdots,|x_r-y_r|^{\frac{1}{\alpha_r}}\},
\]
for  all $x=(x_1,\cdots,x_r), y=(y_1, \cdots, y_r)\in\R^n$. 

Notice that $|x|_s\le |x|\le \sqrt{r}   \, |x|_s$  for all $x\in
\R^n$.
 Using this, one  can  verify the following elementary lemma,
 whose proof is left to the reader.

 \b{Le}\label{norms}
 {For all $x, y\in \R^n$ we have
 $D_s(x,y)\le D_e(x,y)\le r^{1/{2\alpha_1}} D_s(x,y)$.}
\end{Le}

  In general,  $D_s$ does not satisfy the triangle inequality.
  However,
  for each $0<\epsilon\le \alpha_1$, the function $D_s^\epsilon$
is always a metric, called a parabolic  visual metric. In this paper
we consider the following parabolic visual metric
\[
  D(x,y)=D^{\alpha_1}_s(x,y)=
    \max\{|x_1-y_1|,|x_2-y_2|^{\alpha_1/\alpha_2},\cdots,|x_r-y_r|^{\alpha_1/\alpha_r}\}.
\]
With respect to this metric the rectifiable curves in $\R^n\approx
\p G_A\setminus\{\xi_0\}$ are necessarily curves of the form
$\gamma:I\to\R^n$ with $\gamma(t)=(\gamma_1(t),c_2,\cdots,c_r)$
where $c_i\in\R^{n_i}$, $2\le i\le r$, are constant vectors. This
follows from the fact that the directions corresponding to
$\R^{n_i}$, $i\ge 2$, have their Euclidean distance components
``snowflaked'' by the power $\alpha_1/\alpha_i<1$.

\section{Quasisymmetric maps preserve horizontal foliations}\label{hor}


In this section we show that every
self-quasisymmetry of  $\p G_A$ fixes the point $\xi_0\in\p G_A$ and
preserves a natural  foliation on $\p G_A\backslash\{\xi_0\}$.

Recall that a metric space $X$ endowed with a Borel measure $\mu$ is
an \emph{Ahlfors Regular space of dimension Q} (for short,  a
$Q$-regular space) if there exists a constant $C_0\ge 1$ so that
\[
  C_0^{-1} r^Q\le \mu(B_r)\le C_0 r^Q
\]
for every ball $B_r$ with radius  $r<{\text{diam}} (X)$.

We need the following result; see~\cite{T1} for the definition of
the modulus $\text{Mod}_Q$ of a family of curves.

\begin{theorem}[\text{\cite[Theorem~1.4]{T1}}]\label{tyson}
Let $X$ and $Y$ be locally compact, connected, $Q$-regular metric
spaces ($Q>1$)  and let $f:X\ra Y$ be an $\eta$-quasisymmetric
homeomorphism.  Then there is a constant $C$ depending only on
$\eta$, $Q$  and the regularity constants of $X$ and $Y$ so that
\[
   \frac{1}{C}\,\text{Mod}_Q \Gamma \le
           \text{Mod}_Q f(\Gamma)\le C\, \text{Mod}_Q \Gamma
\]
for every curve family $\Gamma$ in $X$. 
\end{theorem}

 Recall that we write $\R^n$ as
   $\R^n=\R^{n_1}\times \cdots \times
\R^{n_r}$. Set $Y=\R^{n_2}\times \cdots \times \R^{n_r}$ and write
$\R^n=\R^{n_1}\times Y$. Since we assume $r\ge 2$, the set $Y$ is
nontrivial.
     The subsets $\{\R^{n_1}\times \{y\}: y\in Y\}$ form a
foliation of $\R^n$. We call this foliation the horizontal foliation
and each leaf $\R^{n_1}\times\{y\}$ a horizontal leaf.
Since $\frac{\alpha_1}{\alpha_i}<1$ for all $2\le i\le r$, we notice
that a curve in $(\R^n,  D)$ is not rectifiable if it is not
contained in a horizontal leaf.

Observe that $(\R^{n_i}, |\cdot|^{\alpha_1/\alpha_i})$ with
the Hausdorff measure (which is comparable to the $n_i$-dimensional Lebesgue
measure) is $n_i\alpha_i/\alpha_1$-regular. Let
$\mu$ be the product of the Hausdorff measures on the factors
$(\R^{n_i},|\cdot|^{\alpha_1/\alpha_i})$. Then it is easy
to see that $(\R^n,D)$ with the measure $\mu$ is $Q$-regular with
$Q=\Sigma_{i=1}^r  n_i\frac{\alpha_i}{\alpha_1}$. It follows that
Theorem~\ref{tyson} applies to the metric space $(\R^n,D)$. We also
point out here that the Hausdorff measure $\mu$ is comparable to the
canonical $n$-dimensional Lebesgue measure on $\R^n$.

\begin{theorem}\label{fixed}
If $r\ge 2$, then every quasisymmetry
$F:(\R^n, D)\ra (\R^n, D)$
preserves the horizontal foliation on $\R^n$.
\end{theorem}

\begin{proof}
    Suppose $F$ does not preserve the horizontal foliation.  Then there
are two points $p$ and $q$ in some $\R^{n_1}\times\{y\}$ such that
$f(p)$ and $f(q)$ are not in the same horizontal leaf.
 Let $\gamma$ be the Euclidean line segment from $p$ to $q$
and $\Gamma$  be the family of straight segments parallel to
$\gamma$ in $\R^n$ whose union is an $n$-dimensional circular
cylinder with $\gamma$ as the central axis. The curves in $\Gamma$
are rectifiable with respect to the metric $D$. Since $f$ is a
homeomorphism, by choosing the radius of the circular cylinder to be
sufficiently small (by a compactness argument) we may assume that no
curve in $\Gamma$ is mapped into a horizontal leaf. It follows that
$f(\Gamma)$ has no locally rectifiable curve and so ${\text{Mod}}_Q
f(\Gamma)=0$. On the other hand, \cite{Va},~7.2 (page~21) shows that
${\text{Mod}}_Q \Gamma>0$ (the Euclidean length element on each
$\beta\in\Gamma$ is the same as the length element on $\beta$
obtained from the metric $D$). Since $Q=\Sigma_{i=1}^r
n_i\frac{\alpha_i}{\alpha_1}>1$, this contradicts Theorem~\ref{tyson}.
Hence each horizontal leaf is mapped to a horizontal leaf.
\end{proof}

\section{Quasisymmetry implies Bilipschitz}\label{bilip}

In this section we show that each self quasisymmetry of $(\R^n,D)$
is actually a bilipschitz map. One should contrast this with the
case of  Euclidean spaces and Heisenberg groups, where there are
non-bilipschitz quasisymmetric maps  (\cite{GV}, \cite{B}). On the
other hand, $(\R^n,D)$ is not as rigid as the ideal boundary of a
quaternionic hyperbolic space or a Cayley plane (\cite{P2}) or a
Fuchsian building (\cite{BP}, \cite{X}), where each self
quasisymmetry is a conformal map.

Let $K\ge 1$ and $C>0$. A bijection $F:X_1\ra X_2$ between two metric spaces is
called a $K$-quasisimilarity (with constant $C$) if
\[
   \frac{C}{K}\, d(x,y)\le d(F(x), F(y))\le C\,K\, d(x,y)
\]
for all $x,y \in X_1$.
It is clear that a map is a quasisimilarity if and only if it is a
bilipschitz map. The point of using the notion of
quasisimilarity is that sometimes there is control on $K$ but not on $C$.

\begin{theorem}\label{main}
Let $F:(\R^n, D)\ra (\R^n, D)$ be an $\eta$-quasisymmetry. Then $F$
is a $K$-quasisimilarity with $K=(\eta(1)/\eta^{-1}(1))^{2r+2}$.
\end{theorem}

In this section, we first develop some intermediate results, and then use
these results to provide a proof of this theorem.
We first recall some definitions.

Let $g: X_1\ra X_2$ be a homeomorphism between two  metric spaces.
We define for every $x\in X_1$  and $r>0$,
\begin{align*}
   L_g(x,r)&=\sup\{d(g(x), g(x')):   d(x,x')\le r\},\\
   l_g(x,r)&=\inf\{d(g(x), g(x')):   d(x,x')\ge r\}.
\end{align*}
Notice that if $X_1$ is
  connected and $X_1\setminus B(x,r)$ is non-empty, then
  $l_g(x,r)\le L_g(x,r)$.   In this paper, we only consider
  connected metric spaces.
  Set
\[
   L_g(x)=\limsup_{r\ra 0}\frac{L_g(x,r)}{r}, \ \ \
   l_g(x)=\liminf_{r\ra 0}\frac{l_g(x,r)}{r}.
\]
  Then
\[
  L_{g^{-1}}(g(x))=\frac{1}{l_g(x)}\ \ \text{ and }\ \ l_{g^{-1}}(g(x))=\frac{1}{L_g(x)}
\]
for any $x\in X_1$. If $g$ is an $\eta$-quasisymmetry, then $L_g(x,r)\le \eta(1)l_g(x, r)$
for all $x\in X_1$ and $r>0$. Hence if in addition
\[
    \lim_{r\ra 0}\frac{L_g(x,r)}{r}\ \ {\text{or}} \ \ \lim_{r\ra 0}\frac{l_g(x,r)}{r}
\]
exists, then
\[
    0\le l_g(x)\le L_g(x)\le \eta(1) l_g(x)\le \infty.
\]

Recall the decomposition $\R^n=\R^{n_1}\times Y$. Given points $y=(x_2,\cdots, x_r)$
and $y'=(x_2', \cdots, x_r')\in Y$ with $x_i,x_i'\in\R^{n_i}$, set
\[
   D_Y(y,y')=\max\{|x_2-x_2'|^{\frac{\alpha_1}{\alpha_2}},
               |x_3-x_3'|^{\frac{\alpha_1}{\alpha_3}}, \cdots,
                      |x_r-x_r'|^{\frac{\alpha_1}{\alpha_r}}\}.
\]
For $p=(x_1,y), p'=(x_1',y')\in \R^{n_1}\times Y$, we have
$D(p,p')=\max\{|x_1-x_1'|, D_Y(y,y')\}$.
We notice that for every $y_1, y_2\in Y$, the Hausdorff distance in the metric $D$
of the two horizontal leaves,
\begin{equation}\label{eq:1}
  HD(\R^{n_1}\times \{y_1\}, \R^{n_1}\times \{y_2\})=D_Y(y_1,  y_2).
\end{equation} 
Also, for any $p=(x_1, y_1)\in \R^{n_1}\times Y$ and any $y_2\in Y$,
\begin{equation}\label{eq:2}
   D(p, \R^{n_1}\times \{y_2\})=D_Y(y_1,y_2).
\end{equation}

By Theorem \ref{fixed} the quasisymmetry $F$ preserves the
horizontal foliation. Hence it induces a map $G: Y\ra Y$  such that
for any $y\in Y$, $F(\R^{n_1}\times \{y\})=\R^{n_1}\times \{G(y)\}$.
For each $y\in Y$, let $H(\cdot,y):\R^{n_1}\ra
\R^{n_1}$  be the map  such that $F(x,y)=(H(x,y),G(y))$ for all
$x\in \R^{n_1}$. Because $F:(\R^n,D)\to(\R^n,D)$ is an
$\eta$-quasisymmetry,
it follows that for each fixed $y\in Y$, the map $H(\cdot, y):\R^{n_1}\ra \R^{n_1}$ is an
$\eta$-quasisymmetry with respect to the Euclidean metric on $\R^{n_1}$.
The following lemma together with equations~(\ref{eq:1}) and~(\ref{eq:2})
imply that $G:(Y, D_Y)\ra (Y, D_Y)$ is also an
$\eta$-quasisymmetry.

\begin{Le}\label{tyson2} \e{(\cite[Lemma 15.9]{T2})}
Let $g: X_1\ra X_2$  be an $\eta$-quasisymmetry and $A,B, C\subset
X_1$.  If  $HD(A,B)\le t\, HD(A, C)$ for some $t\ge 0$, then
there is some $a\in A$ such that
\[
   HD(g(A), g(B))\le \eta(t) d(g(a), g(C)).
\]
\end{Le}

 We recall that if $g: X_1\ra X_2$ is an $\eta$-quasisymmetry, then
$g^{-1}:X_2\ra X_1$ is an $\eta_2$-quasisymmetry, where
$\eta_2(t)=(\eta^{-1}(t^{-1}))^{-1}$, see~\cite[Theorem~6.3]{V2}.
Note that $\eta_2(1)=1/\eta^{-1}(1)$ and $\eta_2^{-1}(1)=1/\eta(1)$.

 In the proofs of the following lemmas, the
quantities $l_G, L_G, l_{G^{-1}}, L_{G^{-1}}$ are defined with
respect to the metric $D_Y$. Similarly, 
$l_{H(\cdot,y)}$, $L_{H(\cdot,y)}$, $l_{I_y}$ and $L_{I_y}$ are
defined with respect to the Euclidean metric on $\R^{n_1}$, where
$I_y:=H(\cdot,y)^{-1}: \R^{n_1}\ra \R^{n_1}$. Lemmas~\ref{l3} and
\ref{l4}   together verify Theorem~\ref{main} for the case $r=2$. At
the end of this section we will use induction to then complete the
proof of Theorem~\ref{main} for the general case $r\ge 2$.

\begin{Le}\label{l1}
The following holds for all $y\in Y$ and $x\in \R^{n_1}$:
\newline
(1)  $L_G(y,r)\le\eta(1)\, l_{H(\cdot,y)}(x,r)$ for $r>0$;
\newline
(2)  $\eta^{-1}(1)\, l_{H(\cdot, y)}(x)\le l_G(y)\le \eta(1)\, l_{H(\cdot,y)}(x)$;
\newline
(3)  $\eta^{-1}(1)\, L_{H(\cdot, y)}(x)\le L_G(y)\le \eta(1)\, L_{H(\cdot,y)}(x)$.
\end{Le}

\begin{proof}
To prove (1), let $y\in Y$, $x\in \R^{n_1}$ and $r>0$. Let $y'\in Y$
be an arbitrary point with $D_Y(y,y')\le r$ and $x'\in \R^{n_1}$ an
arbitrary point with $|x-x'|\ge r$. Then $D((x,y),(x,y'))\le r\le
D((x,y),(x',y))$. Since $F$ is $\eta$-quasisymmetric, we have
\begin{align*}
  D_Y(G(y), G(y'))\le D(F(x,y), F(x,y'))&\le \eta(1)\, D(F(x,y), F(x',y))\\
                                        &=\eta(1)\,|H(x,y)-H(x', y)|.
\end{align*}
Since $y'$ and $x'$ are chosen arbitrarily and are independent of
each other, the inequality follows.

Next we prove~(2) and~(3).
 Since $Y$ is connected, we have  $l_G(y, r)\le L_G(y, r)$.  Now
  the second inequality of~(2) follows  from~(1).  
  Similarly the second inequality of~(3)
follows from~(1) and the fact that $l_{H(\cdot,y)}(x,r)\le
L_{H(\cdot,y)}(x,r)$.

To prove the first inequalities in~(2) and~(3), observe
that the inverse map $F^{-1}:(\R^n,D)\ra (\R^n,D)$
is an $\eta_2$-quasisymmetry, 
with
\[
  F^{-1}(x,y)=(H(\cdot,G^{-1}(y))^{-1}(x),G^{-1}(y))=(I_{G^{-1}(y)}(x),G^{-1}(y)).
\]
Applying the second inequality of~(2) proven above 
to $I_y$ and $G^{-1}$, we obtain:
\[
  \frac{1}{L_G(y)}=l_{G^{-1}}(G(y))\le \eta_2(1)\cdot  l_{I_{y}}(H(x,y))
      =\frac{1}{\eta^{-1}(1)} \cdot\frac{1}{L_{H(\cdot,y )}(x)},
\]
hence $L_G(y)\ge \eta^{-1}(1) L_{H(\cdot,y)}(x)$, which is the first inequality of~(3).
Similarly, using the second inequality of~(3) we
obtain the first inequality of~(2). 
\end{proof}

When $r=2$, we have $Y=\R^{n_2}$  and
$D_Y=|\cdot|^{\frac{\alpha_1}{\alpha_2}}$.

\begin{Le}\label{verticaldila}
Assume that $r=2$. Then $0<l_G(y)\le L_G(y)\le \eta(1)
l_G(y)<\infty$ for a.e. $y\in Y$ with respect to the Lebesgue
measure on $Y=\R^{n_2}$.
\end{Le}

\begin{proof}  Observe in this case that
$D_Y(y,y')=|y-y'|^{\alpha_1/\alpha_2}$ for $y,y'\in Y=\R^{n_2}$.
Because $G$ is an $\eta$-quasisymmetry with respect to the metric
$D_Y$, it is $\eta_1$-quasisymmetric with respect to the Euclidean
metric, where
$\eta_1(t)=(\eta(t^{\alpha_1/\alpha_2}))^{\alpha_2/\alpha_1}$. Hence
the map $G:(\R^{n_2},|\cdot|)\ra (\R^{n_2},|\cdot|)$ is
differentiable a.e. with respect to the Lebesgue measure. With
$L_G^e,l_G^e$ the distortion quantities of the map $G$ with respect
to the Euclidean metric, the differentiability property of $G$ shows
that $\lim_{r\to0}\frac{L_G^e(y,r)}{r}$ and
$\lim_{r\to0}\frac{l_G^e(y,r)}{r}$ exist. Since
$L_G(y,r)=L_G^e(y,r^{\alpha_2/\alpha_1})^{\alpha_1/\alpha_2}$ and
$l_G(y,r)=l_G^e(y,r^{\alpha_2/\alpha_1})^{\alpha_1/\alpha_2}$, this
implies that both $\lim_{r\ra 0}\frac{L_G(y,r)}{r}$ and $\lim_{r\ra
0}\frac{l_G(y,r)}{r}$ exist for a.e. $y\in Y$. It follows that
\[
  0\le l_G(y)\le L_G(y)\le\eta(1) l_G(y)\le\infty.
\]

Fix $y\in Y$ such that both $\lim_{r\ra 0}\frac{L_G(y,r)}{r}$  and
$\lim_{r\ra 0}\frac{l_G(y,r)}{r}$
exist. We next prove that $L_G(y)\not=0,\infty$. 
Suppose that $L_G(y)=\infty$. Then $l_G(y)=\infty$ and so by
Lemma~\ref{l1}~(2),
$l_{H(\cdot,y)}(x)=\infty$ for all $x\in \R^{n_1}$. Hence
$I_y=H(\cdot, y)^{-1}:\R^{n_1}\ra \R^{n_1}$ has the property that
$L_{I_y}(x)=0$ for all $x\in \R^{n_1}$. This implies that $I_y$ is a
constant map, contradicting the fact that it is a homeomorphism.
Similarly we use Lemma~\ref{l1}~(3) to show that $L_G(y)\not=0$.
\end{proof}

In the next two lemmas we use  the fact that $\eta(1)\ge 1$ and
$0<\eta^{-1}(1)\le 1$.

\begin{Le}\label{l2}
Suppose that $r=2$.
Then, for  a.e.   $y\in Y$, the map $H(\cdot,
y):\R^{n_1}\ra\R^{n_1}$ is an $\eta(1)/\eta^{-1}(1)$-quasisimilarity
with constant $l_G(y)>0$.
\end{Le}

\begin{proof}
By Lemma \ref{l1} (2) we have $l_{H(\cdot, y)}(x)\ge
l_G(y)/{\eta(1)}$. Lemma~\ref{l1}~(3) and Lemma~\ref{verticaldila}
imply that,   for a.e. $y\in Y$,  we have  $l_G(y)>0$  and
    $$L_{H(\cdot,y)}(x)\le L_G(y)/\eta^{-1}(1)\le(\eta(1)/\eta^{-1}(1))\,
l_G(y)$$
    for all $x\in \R^{n_1}$.   Because $\R^{n_1}$ is a geodesic
space, for a.e. $y\in Y$ the map $H(\cdot,y)$ is an
$\eta(1)/\eta^{-1}(1)$-quasisimilarity with
constant $l_G(y)$. 
\end{proof}

\begin{Le}\label{l3}
If $r=2$, then there exists a constant $C>0$ with the following properties:
\begin{enumerate}
\item[(1)]     For each $y\in Y$, $H(\cdot,y)$ is an
$(\eta(1)/\eta^{-1}(1))^4$-quasisimilarity with constant $C$;
\item[(2)]     $G:(Y,D_Y)\ra (Y,D_Y)$ is an
$(\eta(1)/\eta^{-1}(1))^5$-quasisimilarity with constant $C$.
\end{enumerate}
\end{Le}

\begin{proof}   (1)
      Fix any $y_0\in Y$  that satisfies  both  Lemma \ref{verticaldila}
      and
   Lemma \ref{l2}.  
           Set $C=l_G(y_0)$. Let $y\in Y$ be an arbitrary point
           satisfying
both  Lemma \ref{verticaldila}
      and
   Lemma \ref{l2}.
             Fix  $x_0\in \R^{n_1}$
and choose $x\in \R^{n_1}$
such that $|x-x_0|\ge D_Y(y,y_0)$. 
Then
\[
   D((x,y_0),(x_0,y))=D((x,y),(x_0,y))=|x-x_0|.
\]
By choosing $x$ so that in addition $|H(x,y_0)-H(x_0,y)|>D_Y(G(y_0), G(y))$,
by the $\eta$-quasisymmetry of $F$ we have
\begin{align*}
|H(x, y_0)-H(x_0, y)| & =D(F(x, y_0), F(x_0, y))\\
    & \le \eta(1) D(F(x,y), F(x_0,y)) =\eta(1) |H(x, y)- H(x_0, y)|.
\end{align*}
By the choice of $y$  and Lemma \ref{l2}, we have
$$|H(x,y)-
H(x_0,y)|\le(\eta(1)/\eta^{-1}(1)) l_G(y)|x- x_0|.$$
  On the other
hand,
\begin{align*}
|H(x, y_0)-H(x_0, y)| & \ge |H(x,y_0)- H(x_0,y_0)|- |H(x_0,y_0)- H(x_0,y)|\\
               & \ge\frac{l_G(y_0)}{\eta(1)/\eta^{-1}(1)} |x-x_0|-|H(x_0,y_0)- H(x_0,y)|.
\end{align*}
      Combining the above inequalities and letting $|x-x_0|\ra \infty$, we obtain
\b{equation}\label{f1}
 l_G(y)\ge \frac{1}{(\eta(1))^3 (\eta^{-1}(1))^{-2}}
l_G(y_0)
   =\frac{C}{(\eta(1))^3 (\eta^{-1}(1))^{-2}}.
\end{equation}
 Switching the roles of $y_0$ and $y$,  we obtain
   $l_G(y)\le (\eta(1))^3 (\eta^{-1}(1))^{-2} l_G(y_0)$.
     By Lemma \ref{verticaldila},
     we have
\b{equation}\label{f2}
 L_G(y)\le \eta(1) l_G(y)\le (\eta(1))^4
(\eta^{-1}(1))^{-2} C.
\end{equation}

Because $\R^{n_1}$ is a geodesic space, to show that $H(\cdot,y)$ is
a quasisimilarity it suffices to gain control over $l_{H(\cdot,y)}$
 and  $L_{H(\cdot,y)}$.  
   By   (\ref{f2})   and  Lemma~\ref{l1}~(3),
\[
  L_{H(\cdot, y)}(x)\le  L_G(y)/\eta^{-1}(1)\le C (\eta(1))^4
  (\eta^{-1}(1))^{-3}
\]
for all $x\in \R^{n_1}$,    and by  (\ref{f1}) and
Lemma~\ref{l1}~(2),
\[
   l_{H(\cdot, y)}(x)\ge \frac{1}{\eta(1)}l_G(y)\ge \frac{C}{(\eta(1))^4(\eta^{-1}(1))^{-2}}.
\]
for all $x\in\R^{n_1}$.
  Hence for a.e. $y$,  $H(\cdot,y)$ is an
$(\eta(1)/\eta^{-1}(1))^4$-quasisimilarity with constant $C$.
  A  limiting argument shows this is true for all $y$.
   Hence (1) holds.

(2)   Recall that when $r=2$ we have $Y=\R^{n_2}$ and
$D_Y=|\cdot|^{\alpha_1/\alpha_2}$. Hence  to prove (2)
   it suffices
to show that $G:(\R^{n_2},|\cdot|)\ra(\R^{n_2},|\cdot|)$ is a
$K$-quasisimilarity with
$K=(\eta(1)/\eta^{-1}(1))^{5\alpha_2/\alpha_1}$. As observed before,
$G$ is $\eta_1$-quasisymmetric with respect to the Euclidean metric,
where $\eta_1(t)=(\eta(t^{\alpha_1/\alpha_2}))^{\alpha_2/\alpha_1}$.
Because $\R^{n_2}$ is a geodesic space, it suffices to gain control
over $l^e_{G}$ and $L^e_G$, where $l^e_{G}$ and $L^e_G$ are similar
to $l_{G}$ and $L_G$, but with the Euclidean metric instead of the
metric $D_Y$. Because $l^e_{G}(p)=l_{G}(p)^{\alpha_2/\alpha_1}$ and
$L^e_{G}(p)=L_{G}(p)^{\alpha_2/\alpha_1}$, it suffices to gain
control over the quantities $l_{G}$ and $L_G$ in terms of
$(\eta(1)/\eta^{-1}(1))^5$.

Notice that (1) implies
$$\frac{C}{(\eta(1)/\eta^{-1}(1))^4}\le  l_{H(\cdot, y)}(x)
\le L_{H(\cdot, y)}(x)  \le  C(\eta(1)/\eta^{-1}(1))^4$$
  for all $x\in \R^{n_1}$   and all $y\in Y$.
  By Lemma \ref{l1},  for all $y\in Y$ we have
$$\frac{C}{(\eta(1)/\eta^{-1}(1))^5}\le  l_G(y)
\le L_G(y)  \le  C(\eta(1)/\eta^{-1}(1))^5.$$
    Hence  (2)   holds.

\end{proof}


\begin{Le}\label{l4}
  Suppose that $r\ge 2$ and there are constants $K\ge 1$ and $C>0$    with the
following properties:
\begin{enumerate}
\item[(1)]  $G:(Y,D_Y)\ra (Y,D_Y)$ is a
$K$-quasisimilarity with constant $C$;
\item[(2)]   For each $y\in Y$,  $H(\cdot,y)$  is a
$K$-quasisimilarity with constant $C$.
\end{enumerate}
   Then $F$ is an $(\eta(1)/\eta^{-1}(1))K$-quasisimilarity with
constant $C$.  
\end{Le}

\begin{proof}
Let $(x_1,y_1),(x_2,y_2)\in\R^{n_1}\times Y$.
We shall first establish a lower bound for $D(F(x_1,y_1),F(x_2,y_2))$.
If $|x_1- x_2|\le  D_Y(y_1,y_2)$,
then $D((x_1, y_1),(x_2,y_2))=D_Y(y_1, y_2)$ and by (1),
\begin{align*}
 D(F(x_1,y_1), F(x_2, y_2))\ge D_Y(G(y_1),G(y_2))
     \ge \frac{C}{K} D_Y(y_1,y_2)
     =\frac{C}{K} D((x_1,y_1),(x_2,y_2)).
\end{align*}
 If $|x_1-x_2|>D_Y(y_1,y_2)$, then
\[
  D((x_1,y_1),(x_2,y_2))=D((x_1,y_2),(x_2,y_2))=|x_1-x_2|,
\]
and since $F$ is an $\eta$-quasisymmetry, by using~(2),
\begin{align*}
 D(F(x_1,y_1),F(x_2,y_2)) &\ge\frac{1}{\eta(1)} D(F(x_1, y_2),F(x_2, y_2))\\
         &=\frac{1}{\eta(1)} |H(x_1,y_2)- H(x_2,y_2)|\\
         &\ge \frac{C}{\eta(1)K}|x_1- x_2|\\
         &=\frac{C}{\eta(1)K}D((x_1,y_1),(x_2,y_2)).
\end{align*}
Hence we have a lower bound for $D(F(x_1,y_1),F(x_2,y_2))$.

By (1), $G^{-1}:(Y,D_Y)\ra (Y,D_Y)$ is a $K$-quasisimilarity with
constant $C^{-1}$. Similarly, (2) implies that for each $y\in Y$,
$(H(\cdot, y))^{-1}$ is a $K$-quasisimilarity with constant
$C^{-1}$.
      Also recall that $F^{-1}$ is an
$\eta_2$-quasisymmetry and $F$ is an $\eta$-quasisymmetry. Now the
argument in the previous paragraph applied to $F^{-1}$ implies
\[
   D(F^{-1}(x_1, y_1), F^{-1}(x_2, y_2))\ge
         \frac{1}{CK\eta_2(1)}D((x_1,y_1),(x_2,y_2)).
\]
It follows that
\[
  D(F(x_1,y_1),F(x_2,y_2))\le CK\eta_2(1) D((x_1, y_1),(x_2,y_2))
            =\frac{CK}{\eta^{-1}(1)}\, D((x_1, y_1),(x_2,y_2))
\]
for  all $(x_1,y_1),(x_2,y_2)\in \R^n$, completing the proof.
\end{proof}

\begin{proof}[\bf{Proof of Theorem~\ref{main}.}]
We induct on $r$.
Lemmas~\ref{l3} and \ref{l4}  yield the desired result in the case
$r=2$.
Now we assume that $r\ge 3$ and that the Theorem is true for $r-1$.
By Lemma~\ref{tyson2},
     $F$ induces an $\eta$-quasisymmetry
$G:(Y,D_Y)\ra (Y,D_Y)$. It follows that $G$ is
$\eta_1$-quasisymmetric with respect to the metric
$D_Y^{\alpha_2/\alpha_1}$ (and it is easy to verify that this is
indeed a metric), where
$\eta_1(t)=[\eta(t^{\alpha_1/\alpha_2})]^{\alpha_2/\alpha_1}$. We
point out here that for $(x_2,\cdots,x_r),(x_2',\cdots,x_r')\in Y$,
\[
  D_Y((x_2,\cdots,x_r),(x_2',\cdots,x_r'))^{\alpha_2/\alpha_1}=
      \max\{|x_2-x_2'|,|x_3-x_3'|^{\alpha_2/\alpha_3},\cdots,|x_r-x_r'|^{\alpha_2/\alpha_r}\}.
\]
Hence the induction hypothesis applied to
$G:(Y,D_Y^{\alpha_2/\alpha_1})\ra (Y,D_Y^{\alpha_2/\alpha_1})$ shows
that $G$ is an $(\eta_1(1)/\eta_1^{-1}(1))^{2r}$-quasisimilarity
with constant $C$. Therefore $G:(Y,D_Y) \ra (Y,D_Y)$ is a
$K_1$-quasisimilarity with constant $C^{\alpha_1/\alpha_2}$, where
\begin{equation}\label{eq:K1}
K_1=\left(\frac{\eta_1(1)}{\eta_1^{-1}(1)}\right)^{\frac{2r\alpha_1}{\alpha_2}}
   =\left(\frac{\eta(1)}{\eta^{-1}(1)}\right)^{2r}.
\end{equation}
This implies that  $C^{\alpha_1/\alpha_2}/{K_1}\le l_G(y)\le
L_G(y)\le  C^{\alpha_1/\alpha_2}K_1$  for all $y\in Y$.
Now Lemma~\ref{l1} yields
\[
   C^{\alpha_1/\alpha_2}\ \frac{1}{K_1\eta(1)}\le  l_{H(\cdot,y)}(x)
    \le L_{H(\cdot,y)}(x)\le C^{\alpha_1/\alpha_2}\ \frac{K_1}{\eta^{-1}(1)}
\]
for all $y\in Y$ and all $x\in \R^{n_1}$.
Since $\R^{n_1}$ is a geodesic space, for each $y\in Y$ the map
$H(\cdot,y)$ is a $K_1\frac{\eta(1)}{\eta^{-1}(1)}$-quasisimilarity with constant
$C^{\alpha_1/\alpha_2}$. By Lemma~\ref{l4}, the map
$F$ is a $K_1(\frac{\eta(1)}{\eta^{-1}(1)})^2$-quasisimilarity with constant
$C^{\alpha_1/\alpha_2}$. Here $K_1$ is as in~(\ref{eq:K1}).
%
\end{proof}

\section{Parabolic Visual   Metrics}\label{parabolic}

In this section we introduce parabolic visual metrics, discuss their
relation with the visual metrics and give a sufficient condition for
  them  to be doubling.   We then use these results to complete the proof of
Theorem~\ref{intromain}.

Parabolic visual metrics have been defined by Hersonsky-Paulin
(\cite{HP}, see also  \cite{BK}) for ${\text{CAT}}(-1)$ spaces. Here
we formally construct parabolic visual metrics in the setting of
Gromov hyperbolic spaces. Since $G_A$ is Gromov hyperbolic, the
theory developed here is applicable to $\p G_A$ as well.     The
metric $D$ (on $\R^n=\p G_A\setminus\{\xi_0\}$) used in the previous
sections is bilipschitz equivalent with  a parabolic visual metric
constructed in this section,  see the discussion after Proposition
   \ref{para.visual}.


Parabolic visual metric is defined on the one-point complement
of the ideal boundary. The relationship between visual metric and
parabolic visual metric is similar to the relationship between the
spherical metric (on the sphere)
and the Euclidean metric (on the one point complement of the
sphere). See Proposition~\ref{relation} for the precise statement.

 Let $X$ be a $\delta$-hyperbolic proper geodesic metric space for some $\delta\ge 0$.
Let $\xi\in \partial X$ and $p\in X$. Then there exists a ray from
$p$ to $\xi$. Let $\gamma: [0, \infty)\rightarrow X$ be such a ray.
Define $B_\gamma:X\rightarrow\R$ by
$B_\gamma(x)=\lim_{t\rightarrow +\infty}(d(\gamma(t),x)-t)$. The
triangle inequality implies that the limit exists
and that 
$|B_\gamma(x)-B_\gamma(y)|\le d(x,y)$ for all $x, y\in X$.
Note that $B_{\gamma}(\gamma(t_0))=-t_0$ for all $t_0\ge 0$.
Since any two rays $\gamma_1$ and $\gamma_2$
from $p$ to $\xi$ are at Hausdorff distance
at most $\delta$ from each other, we have
$|B_{\gamma_1}(x)-B_{\gamma_2}(x)|\le \delta$ for all $x\in X$.

The \e{Buseman function} $B_{\xi,p}: X\rightarrow  \R$
centered at $\xi$ with base point $p$ is:
\[
   B_{\xi,p}(x)=\sup\{B_\gamma(x):  \gamma \;
            {\text{is a geodesic ray from}} \; p \; {\text{to}}\; \xi\}.
\]
Because $B_\gamma$ is $1$-Lipschitz, $B_{\xi,p}$ is $1$-Lipschitz.
The above discussion shows that $B_\gamma(x)\le B_{\xi,p}(x)\le
B_\gamma(x)+\delta$ for all $x\in X$ and every ray $\gamma$ from $p$
to $\xi$. By Proposition 8.2 of~\cite{GdlH}, there exists a constant
$c=c(\delta)$ such that for any two points $p_1,p_2\in X$,  any
$\xi\in \partial X$ and all $x\in X$  we have
\begin{equation}\label{buse}
    |B_{\xi, p_1}(x)-B_{\xi, p_2}(x)-B_{\xi,p_1}(p_2)|\le c.
\end{equation}

Let  $\epsilon>0$, $p\in X$,  $\xi\in \p X$,  and
$\eta_1\not=\eta_2\in
\partial X\backslash\{\xi\}$. Given a complete
geodesic $\sigma$ from $\eta_1$ to $\eta_2$, let
$H_{\xi,p}(\sigma)=\inf\{B_{\xi,p}(x): x\in \sigma\}$. Define
\[
   D_{\xi, p, \epsilon}(\eta_1, \eta_2)=e^{-\epsilon\,
 H_{\xi,p}(\eta_1,\eta_2)},
\]
where
\[
  H_{\xi, p}(\eta_1,\eta_2)=\inf\{H_{\xi, p}(\sigma):\ \sigma
     \text{ is a complete geodesic from }\  \eta_1\; {\text{to}}\; \eta_2\}.
\]
Since any two complete geodesics from $\eta_1$ to $\eta_2$ are at
most Hausdorff distance $2\delta$ apart,  we have
$H_{\xi,p}(\sigma)-2\delta\le H_{\xi,p}(\eta_1,\eta_2)\le H_{\xi,p}(\sigma)$
for any complete geodesic $\sigma$ from $\eta_1$ to $\eta_2$.

An argument similar to that found in~\cite[p.124]{CDP} shows the following:

\begin{Prop}\label{para.visual}
There exists a constant $\epsilon_0$, depending only on $\delta$,
with the following property.
If $X$ is a $\delta$-hyperbolic proper geodesic metric space,
for each $0<\epsilon\le \epsilon_0$, each $p\in X$ and
each $\xi\in\partial X$ there exists a
metric $d_{\xi,p,\epsilon}$ on $\partial X\backslash\{\xi\}$ such that
$\frac{1}{2} D_{\xi,p,\epsilon}(\eta_1, \eta_2)
\le d_{\xi,p,\epsilon}(\eta_1, \eta_2)\le  D_{\xi,p, \epsilon}(\eta_1, \eta_2)$
for all $\eta_1,\eta_2\in \partial X\backslash \{\xi\}$.
\end{Prop}

The metric  $d_{\xi,p, \epsilon}$  is called a \e{parabolic visual
metric}.  With $X=G_A$, $p=(0,0)$,   by using Lemmas \ref{g1}    and
\ref{quasicen} one can see that $D_{\xi_0, p,1}$ is bilipschitz
equivalent with $D_e$.  It follows from Lemma \ref{norms} and
   Proposition \ref{para.visual}
that $d_{\xi_0, p,\alpha_1}$
 is bilipschitz equivalent with the metric $D$ considered in the
 previous sections.


We next discuss how  $d_{\xi,p,\epsilon}$ varies with
$p$ and $\epsilon$.

\begin{Prop}\label{parameter}
Suppose $X$ is a $\delta$-hyperbolic  proper geodesic metric space. Then
\newline
(1) For any $p_1,p_2\in X$, the identity map
${\text{id}}:(\partial X\backslash \{\xi\},d_{\xi,p_1,\epsilon})
\rightarrow(\partial X\backslash \{\xi\},d_{\xi,p_2,\epsilon})$
is a $K$-quasisimilarity, where $K$ depends only on $\delta$;
\newline
(2) For $0< \epsilon_1, \epsilon_2\le \epsilon_0$, the identity map
${\text{id}}:(\partial X\backslash \{\xi\},d_{\xi, p,\epsilon_1})
\rightarrow  (\partial X\backslash \{\xi\},d_{\xi,p,\epsilon_2})$
is $\eta$-quasisymmetric with
$\eta(t)=2^{1+\frac{\epsilon_2}{\epsilon_1}}t^{\frac{\epsilon_2}{\epsilon_1}}$;
\newline
(3) For any $p_1,p_2\in X$ and any
$0< \epsilon_1,\epsilon_2\le \epsilon_0$, the identity map
${\text{id}}:(\partial X\backslash \{\xi\},d_{\xi,p_1,\epsilon_1})
\rightarrow (\partial X\backslash \{\xi\},d_{\xi,p_2,\epsilon_2}) $
is quasisymmetric.
\end{Prop}

\begin{proof}
To prove~(1) let $\eta_1, \eta_2\in\partial X\backslash \{\xi\}$.  Then
Proposition \ref{para.visual} and   inequality (\ref{buse})  imply
\begin{align*}
 d_{\xi,p_2,\epsilon}(\eta_1,\eta_2) \le  D_{\xi,p_2,\epsilon}(\eta_1,\eta_2)
     =e^{-\epsilon\,H_{\xi,p_2}(\eta_1,\eta_2)}
   & \le e^{-\epsilon H_{\xi,p_1}(\eta_1,\eta_2)+ \epsilon B_{\xi,p_1}(p_2)+c\epsilon}\\
   & \le 2\,e^{c\epsilon}\cdot e^{\epsilon B_{\xi,p_1}(p_2)}\cdot d_{\xi,p_1,\epsilon}(\eta_1,\eta_2).
 \end{align*}
Similarly, we obtain
$d_{\xi, p_2,\epsilon}(\eta_1,\eta_2)
\ge\frac{1}{2\,e^{c\epsilon}}\cdot e^{\epsilon B_{\xi,p_1}(p_2)}\cdot d_{\xi,p_1,\epsilon}(\eta_1,\eta_2)$.
The statement holds with $K=2\,e^{c\epsilon_0}$ and constant $C=e^{\epsilon B_{\xi,p_1}(p_2)}$.

The claim~(2) follows from  Proposition \ref{para.visual}, and~(3) follows from~(1) and~(2).
\end{proof}

We next discuss the relation between the parabolic visual metric and
the visual metric. Recall that there is a constant $\epsilon_1$
depending only on $\delta$ such that for any $p\in X$ and any
$0<\epsilon\le \epsilon_1$, there is a visual metric $d_{p,
\epsilon}$ on $\partial X$ satisfying
\begin{equation}\label{visu}
  \frac{1}{2}e^{-\epsilon (\eta_1|\eta_2)_p}\le d_{p,\epsilon}(\eta_1,\eta_2)
              \le e^{-\epsilon (\eta_1|\eta_2)_p}
\end{equation}
for all $\eta_1,\eta_2\in \partial X$. Here $(\xi|\eta)_p$ denotes
the Gromov product of $\xi$ and $\eta$ based at $p$, and
is defined by 
\[
   (\xi|\eta)_p=\frac12\sup\ \liminf_{i,j\ra \i}  \, (d(p,x_i)+d(p,y_j)-d(x_i,y_j))
\]
where the supremum  is taken over all sequences $\{x_i\}\ra \xi$,
$\{y_i\}\ra \eta$. By the $\delta$-hyperbolicity of $X$,
\begin{equation}\label{gromov.p}
  (\xi|\eta)_p-2\delta\le \liminf_{i,j\ra \i}\,(x_i| y_j)_p\le(\xi|\eta)_p
\end{equation}
for all $p\in X$, all $\xi, \eta\in\p X$ and all sequences
 $\{x_i\}\ra \xi$, $\{y_i\}\ra \eta$; we refer the interested reader to
Chapter~7 of~\cite{GdlH}.

To formulate the relation between visual metric and parabolic visual
metric, we need to recall the notion of metric
inversion and sphericalization. The reader is referred to~\cite{BHX} for more details.

Given a metric space $(X,d)$ and $p\in X$,
there is a metric $d_p$ on $X\backslash\{p\}$ satisfying
\[
   \frac{d(x,y)}{4d(x,p)\, d(y,p)}\le d_p(x,y)\le \frac{d(x,y)}{d(x,p)\, d(y,p)}
\]
for all $x,y\in X\backslash\{p\}$. Furthermore,
the identity map $(X\backslash\{p\},d)\ra (X\backslash\{p\},d_p)$ is
$\eta$-quasim\"obius with $\eta(t)=16t$. We call $d_p$
the \e{metric inversion} of $(X,d)$ at $p$.

Let $X$  be an unbounded metric space and $p\in X$.  Let
$S_p(X)=X\cup \{\infty\}$, where $\infty$ is a point not in $X$. We
define a function $s_p:S_p(X)\times S_p(X)\to [0,\i)$ as follows:
\[
   s_p(x,y)=s_p(y,x)=\begin{cases}
              \frac{d(x,y)}{[1+d(x,p)][1+d(y,p)]}&\text{ if }x,y\in X,\\
              \frac{1}{1+d(x,p)}&\text{ if }x\in X\text{ and }y=\i,\\
              0&\text{ if }x=\i=y.
                      \end{cases}
\]
It was shown in~\cite{BK} that there is a metric
$\widehat{d}_p$  on $S_p(X)$ satisfying
\begin{equation}\label{spherical}
   \frac{1}{4}s_p(x,y)\le \widehat{d}_p(x,y)\le s_p(x,y)\;\;\;
                          {\text{for all}}\;\;\;  x,y\in S_p(X).
\end{equation}
Furthermore, the identity map $(X, d)\ra (X,\widehat{d}_p)$ is
$\eta$-quasim\"obius  with $\eta(t)=16t$. We call $\widehat{d}_p$
the \e{sphericalization} of $(X,d)$ at $p$.

If $(Y,d)$ is a bounded metric space, and if a metric inversion is
applied to $Y$, followed by an application of sphericalization, the
resulting space is bilipschitz equivalent to $(Y,d)$. To be more
precise, let $p\not= q\in Y$; assume $p$ is non-isolated in $Y$ and
let  $f:(Y,d)\rightarrow (S_q(Y\backslash\{p\}),\widehat{(d_p)}_q)$
be the map that is identity on $Y\backslash\{p\}$ with $f(p)=\infty$.  
        Then $f$ is bilipschitz (see for example~\cite[Proposition~3.9]{BHX}).

We need the following  result for the proof of Proposition~\ref{relation}.

\b{theorem}[{\text{\cite[Chapter~8]{CDP}}}]\label{treel} 
Let $(Y,h)$ be a $\delta$-hyperbolic space, $y_0\in Y$, and
$Y_0=\{y_0,y_1,\cdots,y_n\}$ be a set of $n+1$ points in $Y\cup\p Y$.
For each $1\le i\le  n$, let $[y_0,y_i]$ be a fixed geodesic
connecting $y_0$ and $y_i$.
Let $X$ denote the union of the geodesics $[y_0,y_i]$,
and choose a positive integer $k$ such that $2n\le 2^k+1$. Then
there exists  a simplicial tree, denoted $T(X)$, and a continuous
map $u:X\ra T(X)$ which satisfies the following properties:
\begin{enumerate}
\item[(i)] For each $i$, the restriction of $u$ to the geodesic $[y_0,y_i]$ is
           an isometry;
\item[(ii)] For every $x$ and $y$ in $X$ we have
           $h(x,y)-2k\delta\le d(u(x), u(y))\le h(x,y)$,
            where $d$ is the  metric on $T(X)$.
\end{enumerate}
\end{theorem}

\begin{Prop}\label{relation}
Let $X$ be a  $\delta$-hyperbolic proper geodesic metric space,
$\xi\in \partial X$, $p\in X$ and $0<\epsilon\le\min\{\epsilon_0, \epsilon_1\}$.
\newline
(1) 
The identity map
\[
  id:(\partial X\backslash\{\xi\},d_{\xi,p,\epsilon})\rightarrow
                        (\partial X\backslash\{\xi\},(d_{p,\epsilon})_\xi)
\]
is  $L$-bilipschitz, where $L$ is a constant depending only on
$\delta$. In particular, 
the parabolic visual metric and the metric inversion of the visual metric
about the point $\xi$ are bilipschitz equivalent;
\newline
(2) 
   Assume $\xi$ is non-isolated in $\p X$.
Let $\eta\in \p X\backslash\{\xi\}$ and
\[
  f:(\p X,d_{p,\epsilon})\rightarrow(S_\eta(\p X\backslash\{\xi\}),\widehat {(d_{\xi,p,\epsilon})}_\eta)
\]
be the bijection that is identity on $\p X\backslash\{\xi\}$ and maps $\xi$ to $\infty$.
Then $f$ is bilipschitz. In particular, the visual metric and the sphericalization of
the parabolic visual metric are bilipschitz equivalent.
\end{Prop}

\begin{proof}
Let $D=d_{p,\epsilon}$ denote the visual metric. Then
$(d_{p,\epsilon})_\xi=D_\xi$.

We first prove~(1). 
Let $\eta_1, \eta_2\in \p X\backslash\{\xi\}$.
By Proposition~\ref{para.visual} and inequality~(\ref{visu}),
\begin{align*}
   \frac{D_\xi(\eta_1,\eta_2)}{d_{\xi,p,\epsilon}(\eta_1,\eta_2)}
    &\le\frac{d_{p,\epsilon}(\eta_1,\eta_2)}{d_{p,\epsilon}(\xi,\eta_1) d_{p,\epsilon}(\xi,\eta_2)}
                \ 2\, e^{\epsilon H_{\xi,p}(\eta_1,\eta_2)} \\
    &\le e^{-\epsilon (\eta_1|\eta_2)_p} \ 2\,
          e^{\epsilon (\xi|\eta_1)_p}\ 2\,e^{\epsilon (\xi|\eta_2)_p}
          \ 2\, e^{\epsilon H_{\xi,p}(\eta_1,\eta_2)}\\
    &=8\, e^{\epsilon\{H_{\xi,p}(\eta_1,\eta_2)+(\xi|\eta_1)_p+(\xi|\eta_2)_p-(\eta_1|\eta_2)_p\}}.
\end{align*}
Similarly,
\[
  \frac{D_\xi(\eta_1,\eta_2)}{d_{\xi, p, \epsilon}(\eta_1,\eta_2)}
       \ge \frac{1}{8}\,e^{\epsilon\{H_{\xi,p}(\eta_1,\eta_2)
             +(\xi|\eta_1)_p+(\xi|\eta_2)_p-(\eta_1|\eta_2)_p\}}.
\]
Now~(1) follows from the following claim.

\noindent{\bf{Claim}}: There is a constant $C$ depending only on
$\delta$ such that if $\eta_1, \eta_2, \xi \in \p X$ are pairwise
distinct,   then
$|H_{\xi,p}(\eta_1,\eta_2)+(\xi|\eta_1)_p+(\xi|\eta_2)_p-(\eta_1|\eta_2)_p|\le
C$.

We now prove the claim. Let $\gamma$ be a ray from $p$ to $\xi$.
Pick a point $y_0\in \gamma$ that is far away from any complete
geodesic joining $\eta_1$ and $\eta_2$. Let $\gamma_i$ ($i=1,2$) be
a ray from $y_0$ to $\eta_i$. Set $X=\gamma \cup
\gamma_1\cup\gamma_2$. By Theorem~\ref{treel} (with the choice
$k=3$) there is a tree $T:=T(X)$ and a map $u:X\ra T$ with the
properties stated in Theorem~\ref{treel}. Let $y'_0,p'\in T$ and
$\xi',\eta'_1,\eta'_2\in \partial T$ be the points corresponding to
$y_0$, $p$, $\xi$ $\eta_1$  and $\eta_2$ respectively. Also let $x'$
be the branch point of $\xi'\eta'_1$ and $\xi'\eta'_2$, and let $y'$
be the projection of $p'$ onto the tripod $Y:=x'\xi'\cup
x'\eta'_1\cup x'\eta'_2$. Let $y \in \gamma$ be the point on
$\gamma$ that is  mapped to $y'$ by $u$ (by choosing $y_0$ far away
from $p$ we may assume that $y$ lies between $p$ and $y_0$).
      Similarly let $x_i\in \gamma_i$ be the point
mapped to $x'$ by $u$. Let $\sigma$ be a complete geodesic from
$\eta_1$ to $\eta_2$. Because $X$ is $\delta$-hyperbolic, geodesic
triangles in $X\cup\p X$ are $24\delta$-thin.
  Also notice that the union $x'\eta'_1\cup x'\eta'_2$ is  a  complete
  geodesic in $T$.
 Now the properties of the map $u$
   given by  Theorem~\ref{treel}
 imply  that the Hausdorff
distance between $\sigma$ and $x_1\eta_1\cup x_2\eta_2$ is bounded
above by a constant $c_1=c_1(\delta)$.

Choose $z_j\in \gamma_1$ and $w_j\in \gamma_2$
with $z_j\rightarrow \eta_1$ and $w_j\rightarrow \eta_2$.  Then the
property of the map $u$ and inequality~(\ref{gromov.p})
imply that $|(\eta_1|\eta_2)_p-(\eta'_1|\eta'_2)_{p'}|\le 11\delta$.
Notice that on the tree $T$ we have $(\eta'_1|\eta'_2)_{p'}=d(p',\eta'_1\eta'_2)$.
Hence  $|(\eta_1|\eta_2)_p-d(p',\eta'_1\eta'_2)|\le 11\delta$.
Similar inequalities hold for $(\xi|\eta_1)_p$ and $(\xi|\eta_2)_p$.

Since the Hausdorff distance between $\sigma$ and
$x_1\eta_1\cup x_2\eta_2$ is at most $c_1$, the definition
of $H_{\xi,p}(\sigma)$  and the property of the map $u$ imply
that $|H_{\xi,p}(\sigma)-H_{\xi',p'}(\eta'_1, \eta'_2)|\le c_1+13\delta$.
The discussion about $H_{\xi,p}(\eta_1, \eta_2)$ shows that
$|H_{\xi,p}(\eta_1,\eta_2)-H_{\xi,p}(\sigma)|\le 2\delta$. It follows that
$|H_{\xi,p}(\eta_1, \eta_2)-H_{\xi',p'}(\eta'_1,\eta'_2)|\le c_1+15\delta$.
Now on the tree $T$, by considering three cases depending on whether
$y'\in x'\xi'$, $y'\in x'\eta'_1$ or $y'\in x'\eta'_2$, we can verify that
\[
    H_{\xi',p'}(\eta'_1,\eta'_2)+d(p',\xi'\eta'_1)+d(p',\xi'\eta'_2)-d(p',\eta'_1\eta'_2)=0.
\]
Now the claim follows by combining the above estimates.

We now prove~(2). By~(1), the identity map
\[
  id:(\p X\backslash\{\xi\},d_{\xi,p,\epsilon})\rightarrow(\p X\backslash\{\xi\},D_\xi)
\]
is bilipschitz.
Pick $\eta\in\partial X\backslash \{\xi\}$.
Then  the map $id$  extends to  a map $F$ between their sphericalizations
\[
  F: (S_{\eta}(\p X\backslash\{\xi\}),\widehat{(d_{\xi,p,\epsilon})}_{\eta})
              \rightarrow (S_{\eta}(\p X\backslash\{\xi\}),\widehat{(D_\xi)}_{\eta}).
\]
Since $id $ is bilipschitz, inequality~(\ref{spherical}) can be used
on $\widehat{(d_{\xi,p,\epsilon})}_{\eta}$ and
$\widehat{(D_\xi)}_{\eta}$ to verify that $F$ is also bilipschitz.
On the other hand,   the natural identification between $(\p X,
d_{p,\epsilon})$ and $(S_{\eta}(\p
X\backslash\{\xi\}),\widehat{(D_\xi)}_{\eta})$   is   bilipschitz.
The statement now follows.
\end{proof}

We next give a sufficient condition for the parabolic visual metric
to be doubling. Recall that a metric space is \e{doubling} if there is a constant $N$ such that
every  open ball with radius $R>0$ can be covered by at most $N$
open balls with radius $R/2$. By a theorem of Assouad  (\cite{A}), a metric space
is doubling if and only if the metric space admits a
quasisymmetric embedding into some Euclidean space.

A metric space X has \e{bounded growth at some scale}, if there are
constants $r,R$  with  $R> r>0$, and an integer $N\ge 1$ such that
every open ball of radius $R$ in $X$ can be covered by $N$ open
balls of radius $r$.

  The following is a consequence of Proposition~\ref{relation},
a result of Bonk-Schramm 
  and  Assouad's  theorem.

\begin{theorem}\label{doubling}
Let $X$ be a $\delta$-hyperbolic geodesic metric space with bounded
growth at some scale. Then for any $\xi\in \partial X$, $p\in X$ and
any $0<\epsilon\le\epsilon_0$, the metric space $(\p
X\backslash\{\xi\},d_{\xi,p,\epsilon})$ is doubling.
\end{theorem}

\begin{proof}
Under the assumption of the Theorem,
Bonk-Schramm has proved that the ideal boundary with the visual
metric is doubling (\cite[Theorem~9.2]{BS}). Hence there is a
quasisymmetric embedding
 $f:(\partial X, d_{p, \epsilon})\rightarrow \R^n$ for some $n\ge 1$.
By Lemma~\ref{inversion} below, $f:(\p
X\backslash\{\xi\},(d_{p,\epsilon})_\xi)\rightarrow
(\R^n\backslash\{f(\xi)\},|\cdot|_{f(\xi)})$ is also a
quasisymmetric  embedding, where $|\cdot|$ denotes the  Euclidean
metric. However, the metric inversion of the Euclidean space is
still a Euclidean space (with one point removed). Hence $(\p
X\backslash\{\xi\},(d_{p,\epsilon})_\xi)$ admits a quasisymmetric
embedding into a Euclidean space, and so is doubling.
Since doubling is invariant
under bilipschitz map, the theorem now follows from
Proposition~\ref{relation} (1).
\end{proof}

Recall that a homeomorphism  $f:X\to Y$ between two metric spaces is
$\eta$-quasim\"obius   for some
  homeomorphism $\eta: [0, \i)\ra [0, \i)$,
if for every four distinct points $x_1,x_2,x_3,x_4\in X$, we have
\[
  \frac{d(f(x_1),f(x_3))\, d(f(x_2),f(x_4))}{d(f(x_1),f(x_4))\,d(f(x_2),f(x_3))}
       \le\eta\left(\frac{d(x_1,x_3)\,d(x_2,x_4)}{d(x_1,x_4)\,d(x_2,x_3)}\right).
\]

\begin{Le}\label{inversion}
Suppose that $f:(X,d)\rightarrow (Y,d)$ is a quasisymmetric embedding.
Then for any $p\in X$,
$f:(X\backslash\{p\},d_p)\rightarrow (Y\backslash\{f(p)\},d_{f(p)})$
is also a quasisymmetric embedding.
\end{Le}

\begin{proof}
Suppose $f$ is an $\eta$-quasisymmetric embedding for some
$\eta$. 
Then $f$ is an  $\tilde\eta$-quasim\"obius embedding for some $\tilde \eta$
depending only on $\eta$, see~\cite[Theorem~6.25]{V2}. 
Now let $x,y,z\in X\backslash\{p\}$ be three
distinct points. Set $q=f(p)$. We calculate
\begin{align*}
 \frac{d_q(f(x),f(z))}{d_q(f(y),f(z))}
  &\le\frac{d(f(x),f(z))}{d(f(x),f(p))\, d(f(z),f(p))}\cdot
         \frac{4\, d(f(y),f(p))\, d(f(z),f(p))}{d(f(y),f(z))} \\
   &=4\ \frac{d(f(x),f(z))\, d(f(y),f(p))}{d(f(x),f(p))\, d(f(y),f(z))}.
 \end{align*}
Similarly,
\[
  \frac{d_p(x,z)}{d_p(y,z)}\ge \frac{1}{4}\ \frac{d(x,z)\, d(y,p)}{d(x,p)\, d(y, z)}.
\]
It follows that
\[
   \frac{d_q(f(x),f(z))}{d_q(f(y),f(z))}\le 4\ \tilde\eta\left(4\,\frac{d_p(x,z)}{d_p(y,z)}\right).
\]
Hence $f:(X\backslash\{p\},d_p)\rightarrow(Y\backslash\{f(p)\},d_{f(p)})$
is $\eta'$-quasisymmetric with $\eta'(t)=4\,\tilde\eta(4t)$.
\end{proof}

Let $(X,d)$ be an unbounded complete metric space  with an Ahlfors $Q$-regular ($Q>1$)
Borel measure $\mu$, and $p\in X$. On the sphericalization
$(S_p(X),\widehat{d}_p)$ we define a measure $\mu'$ as follows:
$\mu'(\{\infty\})=0$, and
on $X=S_p(X)\backslash\{\infty\}$, $\mu'$ is absolutely continuous with
respect to $\mu$ with Radon-Nikodym derivative
\[
  \frac{d\mu'}{d\mu}(x)=\frac{1}{(1+d(p, x))^{2Q}}
\]
for $x\in X$.   
  It can be shown that $(S_p(X),\widehat{d}_p)$ with $\mu'$ is also
$Q$-regular.  

\begin{proof}[{\text{\bf{Proof of   Theorem~\ref{intromain}}}}]
Let $F: (\p G_A,d_{p,\epsilon})\ra (\p G_A,d_{p,\epsilon})$ be a
quasisymmetric map, where $d_{p, \epsilon}$ is a visual metric
($p\in G_A$ and $\epsilon>0$ is sufficiently small). We first prove
that $F(\xi_0)=\xi_0$. Let $D$
  be the metric on $\p G_A\backslash\{\xi_0\}=\R^n$  considered in the previous
  sections.   We have observed that $D$ is bilipschitz equivalent
  with a parabolic visual metric on $\p G_A\backslash\{\xi_0\}$.
  Let  $\theta\in\p G_A\backslash\{\xi_0\}=\R^n$.
Proposition~\ref{relation}  implies that  the natural identification
\[
  (S_\theta(\R^n),\widehat{D}_\theta)
     =(S_\theta(\p G_A\backslash\{\xi_0\}),\widehat{D}_\theta)\ra(\p G_A,d_{p, \epsilon})
\]
is bilipschitz. It follows that (after the above natural
identification)
\[
  F:(S_\theta(\R^n),  \widehat{D}_\theta)\ra 
  (S_\theta(\R^n),  \widehat{D}_\theta)
\]
is quasisymmetric. Let $\mu$ be the product of the Hausdorff measures on the factors
$(\R^{n_i}, |\cdot|^{\frac{\alpha_1}{\alpha_i}})$ of $\R^n$.
Since the metric measure space $(\R^n,D,\mu)$ is $Q$-regular with
$Q=\Sigma_{i=1}^r n_i\frac{\alpha_i}{\alpha_1}$, the remark
preceding the proof shows that
the metric measure space $(S_\theta(\R^n),\widehat{D}_\theta,\mu')$ is also
$Q$-regular. Here $\mu'$ is obtained from $\mu$ as described in the remark
preceding the proof. 
Hence Theorem \ref{tyson} applies to the map
$F:(S_\theta(\R^n),\widehat{D}_\theta)\ra (S_\theta(\R^n),\widehat{D}_\theta)$
and the measure $\mu'$.

Suppose $F(\xi_0)\not=\xi_0$. Under the above natural identification,
this means that $F(\infty)\not=\infty$.
Then ${F}^{-1}(\infty)$ lies in exactly one horizontal
leaf. Fix some $y\in Y$ such that $\R^{n_1}\times \{y\}$ does not
contain ${F}^{-1}(\infty)$. Notice that the subset
$(\R^{n_1}\times\{y\})\cup \{\infty\}$ of $S_\theta(\R^n)$ is an
$n_1$-dimensional topological sphere. So $F(\R^{n_1}\times \{y\}\cup
\{\infty\})$ is an $n_1$-dimensional topological sphere in $\R^n$.
Since each horizontal leaf is an $n_1$-dimensional Euclidean space,
the set $F(\R^{n_1}\times\{y\}\cup \{\infty\})$ is not contained in
any horizontal leaf. It follows that as a dense subset of
$F(\R^{n_1}\times\{y\}\cup \{\infty\})$, the set
$F(\R^{n_1}\times\{y\})$ is also not contained in any horizontal
leaf. Hence there are two points $p$ and $q$ in
$\R^{n_1}\times\{y\}$ such that $F(p)$ and $F(q)$ are not in the
same horizontal leaf.

Let $\gamma$ be the Euclidean line segment from $p$ to $q$ and
$\Gamma$ be the family of straight segments parallel to $\gamma$ in
$\R^n$ whose union is an $n$-dimensional circular cylinder $C$ with
$\gamma$ as the central axis. The curves in $\Gamma$ are rectifiable
with respect to the metric $D$. Since $F$ is a homeomorphism, by
choosing the radius of the circular cylinder to be sufficiently
small (by a compactness argument) we may assume that no curve in
$\Gamma$ is mapped into a horizontal leaf and that $F^{-1}(\infty)$
is not in this cylinder. It follows that $F(\Gamma)$ has no locally
rectifiable curve  with respect to $D$.
Now notice that both $C$ and $F(C)$ are compact subsets of  $\R^n$.
Hence the two metrics $D$ and
$\widehat{D}_\theta$ are bilipschitz equivalent on $C$, as well
as on $F(C)$. It follows that $F(\Gamma)$ has no locally
rectifiable curve  with respect to $\widehat{D}_\theta$.
Hence ${\text{Mod}}_Q F(\Gamma)=0$  in the metric measure  space
$(S_\theta(\R^n), \widehat{D}_\theta, \mu')$. Theorem~\ref{tyson}
then implies that ${\text{Mod}}_Q \Gamma=0$ in the metric measure space
$(S_\theta(\R^n), \widehat{D}_\theta, \mu')$.
On the other hand,
${\text{Mod}}_Q \Gamma>0$  in the metric measure space
$(\R^n, D, \mu)$ (see the proof of  Theorem \ref{fixed}).
Since $D$ and $\widehat{D}_\theta$ are bilipschitz equivalent on $C$, and
$\mu$ and $\mu'$ are also comparable on $C$,
we have ${\text{Mod}}_Q \Gamma>0$ in the metric measure space
$(S_\theta(\R^n),\widehat{D}_\theta, \mu')$, a contradiction.
Hence $ F(\xi_0)=\xi_0$.

Next we prove that $F$ is bilipschitz with respect to the metric $D$.
Since the map $F:(\p G_A,d_{p,\epsilon})\ra (\p G_A,d_{p,\epsilon})$
is quasisymmetric, Lemma~\ref{inversion} implies that
\[
  F:(\p G_A\backslash\{\xi_0\}, (d_{p, \epsilon})_{\xi_0})
   \rightarrow(\p G_A\backslash\{\xi_0\}, (d_{p, \epsilon})_{\xi_0})
\]
is also a quasisymmetric map. By Proposition~\ref{relation},
$id:(\p G_A\backslash\{\xi_0\},(d_{p,\epsilon})_{\xi_0})
       \rightarrow(\p G_A\backslash\{\xi_0\},d_{\xi_0,p,\epsilon})$
is bilipschitz, where $d_{\xi_0,p,\epsilon}$ is a parabolic
visual metric. It follows that
\[
   F:(\p G_A\backslash\{\xi_0\},d_{\xi_0,p,\epsilon})
         \rightarrow(\p G_A\backslash\{\xi_0\},d_{\xi_0,p,\epsilon})
\]
is quasisymmetric. By Proposition~\ref{parameter}, any two parabolic
visual metrics are quasisymmetrically equivalent. By the discussion
following Proposition  \ref{para.visual}, it follows that
    $F:(\p G_A\backslash\{\xi_0\}, D)\ra (\p
G_A\backslash\{\xi_0\},D)$ is quasisymmetric. Now the result follows
from Theorem~\ref{main}.
\end{proof}

\section{Consequences}\label{applica}

In this section we will prove the corollaries from the introduction.

We note that because $G_A$ has sectional curvature $-\alpha_r^2\le
K\le -\alpha_1^2$, $G_A$ is a proper geodesic $\delta$-hyperbolic
space with $\delta$ depending only on $\alpha_1$.

\begin{proof}[{\text{\bf{Proof of Corollary~\ref{finitege}.}}}] Suppose there is a
quasiisometry $f:G_A\ra G$ from $G_A$ to a finitely generated group
$G$, where $G$ is equipped with a fixed word metric. Since $G_A$ is
Gromov hyperbolic, it follows that $G$ is Gromov hyperbolic and $f$
induces a quasisymmetric map $\p f: \p G_A\ra \p G$.  The left
translation of $G$ on itself induces an action of $G$ on the Gromov
boundary $\p G$ by quasisymmetric maps.  By conjugating this action
with $\p f$ we obtain an action of $G$ on $\p G_A$ by quasisymmetric
maps.  By Theorem \ref{intromain}, this action has a global fixed
point.  It follows that the action of $G$ on $\p G$ has a global
fixed point. This can happen only when $G$ is virtually infinite
cyclic, in which case the Gromov boundary $\p G$ consists of only
two points. This contradicts the fact that $\p G_A$ is a sphere of
dimension $n\ge 2$ (since $r\ge 2$).
\end{proof}

The proofs of Corollaries   \ref{c2}
    and    \ref{c0} require
some preparation.

Let $X$ be a proper geodesic $\delta$-hyperbolic space and
$\xi_1,\xi_2,\xi_3\in \p X$ be three distinct points in the Gromov
boundary. For any constant $C\ge 0$, a point $x\in X$ is called a
$C$-\e{quasicenter} of the three points $\xi_1,\xi_2,\xi_3$ if for
each $i=1, 2, 3$, there is a geodesic $\sigma_i$ joining $\xi_i$ and
$\xi_{i+1}$ ($\xi_4:=\xi_1$) such that the distance from $x$ to
$\sigma_i$  is at most $C$. For any $C\ge 0$, there is a constant
$C'$ that depends only on $\delta$ and $C$ such that the distance
between any two  $C$-quasicenters of $\xi_1,\xi_2,\xi_3$ is at most
$C'$.

 The following three lemmas hold in all Hadamard manifolds with
 pinched negative sectional curvature.

\b{Le}\label{g1}
  {Let $(x,t)\in G_A=\R^n\times \R$ be an arbitrary point, and
  $\sigma$ a geodesic through $(x,t)$ and tangent to the horosphere
    $\R^n\times \{t\}$. Let $p, q\in \R^n\approx \p G_A\backslash \{\xi_0\}$
      be the two  endpoints of $\sigma$.  Then $(x,t)$ is a
       $12\delta$-quasicenter for $p, q, \xi_0$.}

       \end{Le}

       \b{proof}
As an  ideal geodesic triangle
 in a  $\delta$-hyperbolic space,
$\sigma\cup \gamma_p\cup \gamma_q$
 is
$4\delta$-thin. Hence there is some point
 $m\in \gamma_p\cup \gamma_q$  with  $d((x,t), m)\le 4\delta$. We may assume
   $m=(p, t')\in \gamma_p$ for some $t'\in \R$.  We may further
   assume that $(p, t')$ is the point on $\gamma_p$ nearest to
   $(x,t)$.  Then the geodesic segment from $(x,t)$ to $(p, t')$
   must be perpendicular to the geodesic $\gamma_p$. This implies
   $t'>t$.  Since
  $(x,t)$ is the highest point on $\sigma$  and  is  more than
  $4\delta$ below the horosphere  through $(p, t'+4\delta)$, we have
$d((p, t'+4\delta), \sigma)>4 \delta$.
 Now the thin triangle condition applied to
  the point $(p, t'+4\delta)$ and the triangle
$\sigma\cup \gamma_p\cup \gamma_q$
 implies there is some $(q, t'')\in \gamma_q$ with
  $d((p, t'+4\delta),(q, t''))\le 4\delta$.  The triangle inequality
    together with $d((p,t'),(p,t'+4\delta))=4\delta$
  implies $d((x,t), (q, t''))\le 12 \delta$. Hence $(x,t)$ is a
  $12\delta$-quasicenter for
$p, q, \xi_0$.

       \end{proof}

Let $M$ be a simply connected Riemannian manifold with sectional
curvature
 $-b^2\le K\le -a^2$, where $b>a>0$. For any $\xi\in \p M$, any
 horosphere $\mathcal{H}$ centered at $\xi$, and any two points $x,
 y\in \mathcal{H}$, the distance $d_{\mathcal{H}}(x,y)$ between $x$
 and $y$ in the horosphere is related to $d(x,y)$  by (see \cite{HI}):
\b{equation}\label{pinching} \frac{2}{a}\,\sinh \left(\frac{a}{2}
\,d(x,y)\right)\le d_{\mathcal{H}}(x,y)\le \frac{2}{b}\,\sinh
\left(\frac{b}{2}\, d(x,y)\right).
\end{equation}
 For any $s>0$, let $\mathcal{H}_s$ be the horosphere centered at
 $\xi$ that is closer to $\xi$ than $\mathcal{H}$ and is at distance
 $s$ from $\mathcal{H}$.   Let $\phi_s: \mathcal{H}\ra
 \mathcal{H}_s$ be the map which sends each $x\in \mathcal{H}$ to
 the unique intersection point of $x\xi$ with $\mathcal{H}_s$.
  Then for each tangent vector $v\in T_x{\mathcal{H}}$   of $\mathcal{H}$ at $x$ we have (see
  \cite{HI}):
   $e^{-bs}\lVert v\rVert\le \lVert d\phi_s(v)\rVert \le e^{-as} \lVert v\rVert.$
   It follows that for any rectifiable curve $c$ in $\mathcal{H}$,
    the lengths of $c$ and $\phi(c)$ are related by
     $ e^{-bs}\ell(c)\le \ell(\phi(c))\le e^{-as} \ell(c)$.

\b{Le}\label{quasicen} { Let $p, q\in \R^n\approx \p G_A\backslash
\{\xi_0\}$   and
   suppose  that   $D_e(p,q)=e^{t_0}$.  Then  $(p,t_0)$ is a $C$-quasicenter for $p, q, \xi_0$,
       where $C$ depends only on  $\alpha_1$  and $\alpha_r$.

}
\end{Le}

\b{proof}
    Let $\sigma$ be the geodesic  in $G_A$  joining $p,  q\in \p G_A\backslash
\{\xi_0\}$, and
 $(x,t)$  the highest point on $\sigma$. We may assume $d((x,t),
 \gamma_p)\le 4\delta$. Let $(p, t_1)\in \gamma_p$ be the point
 nearest to $(x,t)$. Then  $t_1>t$
  and the argument in the proof of Lemma~\ref{g1} gives a point
$(q,t_2)\in \gamma_q$ such that   $d((p,t_1),(q,t_2))\le 8\delta$.
It follows that $|t_1-t_2|\le 8\delta$.
       The triangle inequality then implies
        $d((p, t_1), (q, t_1))\le 16 \delta$.

          By the definition of $D_e$, we have
           $d_{\R^n\times \{t_0\}}(p, q)=1$.  Hence  $d((p, t_0), (q, t_0))\le
           1$.   If $t_0\le t_1$, then the convexity of the distance
           function
      $f(t):=d(\gamma_p(t), \sigma)=d((p,t), \sigma)$
           implies $d((p,t_0), \sigma)\le d((p, t_1),
           \sigma)\le 4\delta$. In this case,
             $(p,t_0)$ is a $\max\{1, 4\delta\}$-quasicenter of
             $p,q, \xi_0$.  Now we suppose $t_0>t_1$.
              Join $(p, t_1)$ and $(q, t_1)$ by a shortest path   $c$  in
              the horosphere   $\mathcal{H}:=\R^n \times \{t_1\}$.
   By (\ref{pinching})  we have
                $\ell(c)\le \frac{2}{\alpha_r}\sinh
\left(8\alpha_r \delta\right)$.   The projection
              $\phi_{t_0-t_1}(c)$ is a path in the horosphere
              $\R^n\times \{t_0\}$ joining  $(p, t_0)$ and $(q,
              t_0)$.  Hence
\[1=d_{\R^n\times \{t_0\}}(p,q)\le \ell(\phi_{t_0-t_1}(c))\le
              e^{-(t_0-t_1)\alpha_1}\ell(c)\le \frac{2}{\alpha_r}e^{-(t_0-t_1)\alpha_1}\sinh
\left(8\alpha_r\delta\right). \]
 It follows that
   $t_0-t_1\le  C_1$, where
\[C_1= \frac{\ln[\frac{2}{\alpha_r}\sinh \left(8\alpha_r
\delta\right)]}{\alpha_1}.
\]
 The triangle inequality then implies $d((p, t_0), \sigma)\le
 C_1+4\delta$. Hence $(p, t_0)$ is a $C$-quasicenter for
  $p, q,\xi_0$, where $C=\max\{1, C_1+4\delta\}$.

\end{proof}

\b{Le}\label{g3} {Let $p, q\in \R^n\approx \p G_A\backslash
\{\xi_0\}$   and
 suppose  that  $D_e(p,q)=e^{t_0}$.  \newline
 (1) If  $t_1, t_2<t_0$, then
 $|d((p,t_1),(q,t_2))-(t_0-t_1)-(t_0-t_2)|\le C$,  where $C$ depends
only on $\alpha_1$ and $\alpha_r$;\newline
 (2)  If $t_1\ge t_0$ or $t_2\ge t_0$, then
$|t_1-t_2|\le d((p,t_1),(q,t_2))\le  |t_1-t_2|+ 1$.

}
\end{Le}

\b{proof}
 (1)
Let $\sigma$ be the geodesic in $G_A$ joining $p$ and $q$,  and
$(x,t)$ the highest point on $\sigma$.  By Lemmas \ref{g1} and
\ref{quasicen},  the three points $(p, t_0)$, $(q, t_0)$ and $(x,t)$
are all $c_1$-quasicenters of $\xi_0$,$p$, $q$, where $c_1$ depends
only on $\alpha_1$  and $\alpha_r$. Hence $d((p, t_0), (x,t))\le
c_2$  and $d((q, t_0), (x,t))\le c_2$  for some $c_2=c_2(c_1,
\delta)=c_2(\alpha_1,\alpha_r)$.  Since $t_1<t_0$,
   the convexity of
distance function implies that   $d((p, t_1), m_1)\le d((p,t_0),
(x,t))\le c_2$ for some point $m_1\in \sigma$   lying  between
$(x,t)$ and $p$.
  Similarly, there is some point $m_2\in \sigma$ between $(x,t)$ and
  $q$ with $d((q, t_2), m_2)\le  c_2$.
 By triangle inequality we have $|d((p, t_1), (q,
 t_2))-d(m_1,m_2)|\le 2c_2$.
    Since $d((p,t_0), (x, t))\le c_2$ and $d((p,t_1), m_1)\le c_2$,
    the triangle inequality also implies
    $|d((p,t_0), (p,t_1))-d(m_1, (x,t))|\le 2c_2$. Similarly,
$|d((q,t_0), (q,t_2))-d(m_2, (x,t))|\le 2c_2$.
     Since $d(m_1,m_2)=d(m_1,(x,t))+d((x,t), m_2)$  and $d((p,t_0),
     (p,t_1))=t_0-t_1$, $d((q,t_0),
     (q,t_2))=t_0-t_2$,   the above estimates together yield
$|d((p,t_1),(q,t_2))-(t_0-t_1)-(t_0-t_2)|\le  6 c_2$.

(2) We may assume $t_1\ge t_0$.  Then the convexity of distance
function and the definition of $D_e$  imply
 $$d((p, t_1), (q, t_1))\le d((p, t_0), (q, t_0))\le d_{\R^n\times
 \{t_0\}}((p, t_0), (q, t_0))=1.$$  Now  (2) follows from the
 triangle inequality  and  (\ref{eq:lowerbound}).

\end{proof}

Corollary~\ref{c2} follows from Theorem~\ref{intromain}  and the
following lemma.
  Notice that, by Theorem \ref{intromain}, for any
  quasiisometry   $f: G_A\ra G_A$, the boundary map
  $\p f$ fixes $\xi_0$ and restricts to a homeomorphism of
  $\p G_A\setminus\{\xi_0\}$, which we  still denote by $\p f$.


%

\begin{Le}\label{hightres}
Let $f: G_A\ra G_A$ be a quasiisometry. Then $f$ is
height-respecting if and only if $\p f: (\p G_A\setminus\{\xi_0\},
D)\ra (\p G_A\setminus\{\xi_0\},D)$ is a bilipschitz map.
\end{Le}

\begin{proof}
Dymarz (\cite[Lemma~7]{D}) proved that the boundary map of a
height-respecting quasiisometry is a bilipschitz map with respect to
the quasimetric $D_s$. It follows that the  boundary map is also
bilipschitz with respect to the metric $D$. Hence
we only prove the \lq\lq if" part. 
So we assume
   $\p f$ is bilipschitz w.r.t. $D$.  Notice that it is  also
     bilipschitz w.r.t. $D_e$.   Hence
 there is a constant $L\ge 1$ such that for all
$p,q\in\p G_A\setminus\{\xi_0\}=\R^n$,
\[
   {D_e}(p,q)/L\le {D_e}(\p f(p),\p f(q))\le L {D_e}(p,q).
\]
Let $(x,t)\in G_A=\R^n\times \R$. Pick any geodesic $\sigma$ through
$(x,t)$ that is tangent to the horosphere $\R^n\times \{t\}$.  Then
the two endpoints $p, q$  of $\sigma$ are in $\p
G_A\backslash\{\xi_0\}=\R^n$. If $t_0$ is the real number such that
$d_{\R^n\times \{t_0\}}((p, t_0), (q, t_0))=1$, then by the
definition of $D_e$ we have $D_e(p,q)=e^{t_0}$.   By Lemmas \ref{g1}
and \ref{quasicen} both $(x,t)$ and $(p,t_0)$ are $c_1$-quasicenters
of the three points $p,q,\xi_0\in \p G_A$, where $c_1$ depends only
on $\alpha_1$ and $\alpha_r$.    Hence there is a constant $c_2$
depending only on $c_1$  and $\delta$  such that
$d((x,t),(p,t_0))\le c_2$.
By~(\ref{eq:lowerbound}), we have $|t-t_0|\le c_2$.
Let $t_0'$ be the real number such that $d_{\R^n\times \{t_0'\}}((\p
f(p), t_0'),  (\p f(q),t_0'))=1$. Then $D_e(\p f(p),\p
f(q))=e^{t_0'}$. Since $f$ is a quasiisometry between
$\delta$-hyperbolic spaces, $f(x,t)$ is a $c_3$-quasicenter of $\p
f(p), \p f(q), \p f(\xi_0)=\xi_0$, where $c_3$ depends only on
$\delta$, $c_1$ and the quasiisometry constants of $f$. As $(\p
f(p),t_0')$ is a $c_1$-quasicenter of these three points, we have
$d((\p f(p), t_0'), f(x,t))\le c_4$, with $c_4$
depending only on $c_1$, $c_3$ and $\delta$.   
Let $t'$ be the height of $f(x, t)$. Then 
by~(\ref{eq:lowerbound}) again, $|t'-t_0'|\le c_4$.

 The bilipschitz assumption of $\p f$ and the formulas  $D_e(\p f(p), \p
 f(q))=e^{t_0'}$  and
$D_e(p,q)=e^{t_0}$  imply that $|t_0-t_0'|\le \ln L$. Combining this
with $|t-t_0|\le c_2$ and $|t'-t_0'|\le c_4$, we obtain $|t-t'|\le
\ln L+c_2+c_4$. Hence the heights of any point $(x,t)$ and its image
$f(x,t)$ differ by at most a constant that is independent of
$(x,t)$. The corollary follows.
\end{proof}

\begin{proof}[{\text{\bf{Proof of Corollary~\ref{c0}}}}] Let $f:G_A\ra G_A$
be an $(L,A)$-quasiisometry. By Theorem~\ref{intromain}, the
boundary map $\p f:\p G_A\ra \p G_A$ fixes the point $\xi_0$.
Let $(x_1,t_1),(x_2,t_2)\in\R^n\times\R=G_A$.   Suppose
$D_e(x_1,x_2)=e^{t_0}$.  We only consider the case $t_0>t_1,t_2$,
the other cases being similar.  By Lemma \ref{g3} there is a
constant $c_1=c_1(\alpha_1,\alpha_r)$   such that
\begin{equation}\label{eq:Control}
  |d((x_1,t_1),(x_2,t_2))-(t_0-t_1)-(t_0-t_2)|\le c_1.
\end{equation}

Let $t'_i$ ($i=1,2$) be the height of $f(x_i,t_i)$. 
By Corollary~\ref{c2}, there is a
constant $c_2\ge 0$   
such that $|t_i-t'_i|\le c_2$. Since $f(\ga_{x_i})$ is an $(L,
A)$-quasigeodesic  joining $\xi_0$ and $\p f(x_i)$, there is a
constant $c_3$ depending only on $L$, $A$ and $\delta$ such that the
Hausdorff distance between $f(\ga_{x_i})$ and $\ga_{\p f(x_i)}$ is
at most $c_3$. Hence there is some $t''_i$ such that $d((\p
f(x_i),t''_i),f(x_i,t_i))\le c_3$. It follows that $|t'_i-t''_i|\le
c_3$ and hence $|t_i-t''_i|\le c_2+c_3$.

Suppose   $D_e(\p f(x_1), \p f(x_2))=e^{t'_0}$.
 By Lemma \ref{quasicen}  $(\p f(x_1),t'_0)$  is a $c_4$-quasicenter for
$\xi_0$, $\p f(x_1)$,   $\p f(x_2)$, where $c_4=c_4(\alpha_1,
\alpha_r)$.
  Similarly, $(x_1, t_0)$ is a $c_4$-quasicenter for $\xi_0$, $x_1$,
  $ x_2$.
On the other hand, since $f$ is an $(L, A)$ quasiisometry,
$f(x_1,t_0)$ is a $c_5$-quasicenter of $\xi_0$, $\p f(x_1)$ and $\p
f(x_2)$,  where   $c_5=c_5(L, A, c_4, \delta)$. It follows that
$d((\p f(x_1),t'_0),f(x_1,t_0))\le c_6$ for some constant
$c_6=c_6(c_4, c_5, \delta)$. Let $t''_0$ be the height of $f(x_1,
t_0)$. Then $|t'_0- t''_0|\le c_6$. By Corollary~\ref{c2} we have
$|t_0-t''_0|\le c_2$. 
  Hence  $|t_0-t'_0|\le c_6+c_2$.

  Next we consider two cases:\newline
  {\bf{Case 1}}.  Both $t_1'', t_2''<t_0'$. \newline
  In this case, by Lemma \ref{g3} (1)  again we have
\[
  |d((\p f(x_1),t''_1), (\p f(x_2), t''_2))-(t'_0-t''_1)-(t'_0-t''_2)|\le c_1.
\]
Combining this with~(\ref{eq:Control}) and the estimates
$|t_i-t''_i|\le c_2+c_3$, $|t_0- t'_0|\le c_6+c_2$, and $d((\p
f(x_i), t''_i),f(x_i, t_i))\le c_3$, we obtain
\[
  |d((x_1,t_1),(x_2,t_2))-d(f(x_1, t_1), f(x_2, t_2))|
               \le 2c_1+4c_2+4c_3+2c_6.
\]

\noindent
 {\bf{Case 2}}.  Either $t_1''\ge t_0'$ or $t_2''\ge t_0'$.
\newline
   Without loss of generality,  we may assume $t_1''\ge t_0'$
     and  $t_1''\ge t_2''$.  Then Lemma \ref{g3} (2) implies
\b{equation}\label{co1}
  |d((\p f(x_1),t''_1), (\p f(x_2), t''_2))-(t''_1-t''_2)|\le 1.
\end{equation}
 On the other hand,  $t_1''\ge t_0'$  and the assumption
  $t_0>t_1$  together with
$|t_0-t'_0|\le c_6+c_2$  and  $|t_i-t''_i|\le c_2+c_3$
  imply that
  $|t_0-t_1|\le 2c_2+c_3+c_6$.  Now it follows from
    (\ref{eq:Control})  and the triangle inequality that
\b{equation}\label{co2}
 |d((x_1, t_1), (x_2, t_2))-(t_1-t_2)|\le
c_1+4c_2+2c_3+2c_6.
\end{equation}
 Now (\ref{co1}), (\ref{co2}),
$|t_i-t''_i|\le c_2+c_3$, and $d((\p f(x_i), t''_i),f(x_i, t_i))\le
c_3$  imply
\[
  |d((x_1,t_1),(x_2,t_2))-d(f(x_1, t_1), f(x_2, t_2))|
               \le 1+c_1+6c_2+6c_3+2c_6.
\]

\end{proof}

\noindent Addresses:

\noindent N.S.: Department of Mathematical Sciences, P.O.Box 210025, University of
Cincinnati, Cincinnati, OH 45221-0025, U.S.A.\hskip .4cm
E-mail: nages@math.uc.edu

\noindent X.X.: Dept. of Mathematical Sciences, Georgia Southern University,
Statesboro, GA 30460, U.S.A.\hskip .4cm
E-mail: xxie@georgiasouthern.edu

 \addcontentsline{toc}{subsection}{References}

\end{document}